\documentclass[12pt]{amsart}
\usepackage{graphicx}
\usepackage{color}
\usepackage{amsmath,amssymb,setspace,nicefrac}
\usepackage[active]{srcltx}
\usepackage[colorlinks, linkcolor=red, citecolor=blue, urlcolor=blue, hypertexnames=true]{hyperref}
\usepackage{amsrefs}

\allowdisplaybreaks
\usepackage[all]{xy}

\newcommand{\R}{{\mathbb R}}

\newcommand{\N}{{\mathbb N}}

\newcommand{\cM}{{\mathcal M}}
\newcommand{\cP}{{\mathcal P}}

\newcommand{\fC}{{\mathfrak C}}

\newcommand{\Opt}{{\rm Opt}}

\newcommand{\supp}{{\rm supp}}
\newcommand{\CAT}{{\rm CAT}}
\newcommand{\LET}{{\rm LET}}

\newtheorem{thm}{Theorem}[section]

\newtheorem{lem}[thm]{Lemma}

\newtheorem{definition}[thm]{Definition}
\newtheorem{example}[thm]{Example}
\newtheorem{remark}[thm]{Remark}
\newtheorem{proposition}[thm]{Proposition}

\theoremstyle{definition}

\begin{document}
\title[Barycenters in the Hellinger-Kantorovich space]{Barycenters in the Hellinger-Kantorovich space}
	
	\author{Nhan-Phu Chung}
	\address{Nhan-Phu Chung, Department of Mathematics, Sungkyunkwan University, 2066 Seobu-ro, Jangan-gu, Suwon-si, Gyeonggi-do, Korea 16419.} 
	\email{phuchung@skku.edu;phuchung82@gmail.com} 
	\author{Minh-Nhat Phung}
	\address{Minh-Nhat Phung, Department of Mathematics, Sungkyunkwan University, 2066 Seobu-ro, Jangan-gu, Suwon-si, Gyeonggi-do, Korea 16419. }
	\email{pmnt1114@skku.edu;pmnt1114@gmail.com} 
	\date{\today}
	\maketitle

\begin{abstract}
		Recently, Liero, Mielke and Savar\'{e} introduced the Hellinger-Kantorovich distance on the space of nonnegative Radon measures	of a metric space $X$. We prove that Hellinger-Kantorovich barycenters always exist for a class of metric spaces containing of compact spaces and Polish $\CAT(1)$ spaces; and if we assume further some conditions on the data, such barycenters are unique. We also introduce homogeneous multimarginal problems and illustrate some relations between their solutions and Hellinger-Kantorovich barycenters. Our results are analogous to the work of Agueh and Carlier for Wasserstein barycenters.  	
		
%\keywords{optimal transport, Hellinger-Kantorovich distance, barycenter,  multimarginal problem, duality}
%\subclass{MSC 49J40, 49K21, 49K30}		
		
	\end{abstract}
	\section{Introduction}
	A notion of barycenter in the Wasserstein space over $\R^n$ has been introduced and investigated by Agueh and Carlier in \cite{MR2801182}, and has been explored extensively after that. It is closely related to optimal transport problems and has many applications in other fields such as texture mixing in computer vision \cite{MR3469435}, machine learning \cite{MR3862415}, multi-population matching equilibrium in economics \cite{MR3423268}. The Wasserstein barycenter also has been investigated for measures (not necessarily with finite support) over compact manifolds \cite{MR3590527} and locally compact geodesic spaces \cite{MR3663634}. 
	
	On the other hand, recently unbalanced optimal transport problems and various generalized Wasserstein distances on the space of finite measures have been introduced and studied by various authors \cite{CPSV,KMV,KV, MR3542003,MR3763404, PR14}. In \cite{MR3763404}, Liero, Mielke and Savar\'{e} define the Hellinger-Kantorovich distance $HK(\mu,\nu)$ between measures on a metric space via homogeneous marginals and Wasserstein distances over its Euclidean cone $\fC$. They also represent this distance in terms of Logarithmic Entropy Transport problems and establish a Benamou-Brenier formula for it. As natural we would like to ask whether we can define barycenters and show their existence, uniqueness and consistency in the Hellinger-Kantorovich space as in the Wasserstein setting. Because Hellinger-Kantorovich spaces in general are not NPC, the existence and uniqueness of barycenters do not follow straightforward from the work of \cite{Sturm}.
	
	Let $(X,d)$ be a metric space and $(\fC,d_{\fC})$ be its Euclidean cone. Let $p\geq 2$ be an integer, $\lambda_1,\lambda_2,\ldots,\lambda_p $ be positive real numbers satisfying $ \sum_{i=1}^p\lambda_i=1 $ and $ \mu_1,\mu_2,\ldots,\mu_p \in\cM(X)$. We consider the following Hellinger-Kantorovich barycenter problem 	
	$$\inf_{\mu\in \cM(X)}J(\mu)=\sum_{i=1}^p\lambda_i{HK}^2(\mu,\mu_i) \mbox{   } (\cP). $$
		
	In this article, we will first prove that barycenters in the Hellinger-Kantorovich space $\cM(X)$ always exist if our base space $X$ is a Polish metric space having property $(BC)$ (see definition \ref{D-(BC) property} in section 3). Roughly speaking, our assumptions on metric spaces concern a nice selection of the `pointwise barycenter'. Our class of examples include all compact geodesic spaces, and Polish $\CAT(1)$ spaces. To do that, we lift up a sequence of minimizing measures to their corresponding measures on the cone, and if we can find Wasserstein barycenters for the latter ones we can push back a minimizing solution in Hellinger-Kantorovich distance. As we need the existence of Wasserstein barycenters in the cone $\fC$ which is not locally compact unless $X$ is compact, we also can not apply directly results in \cite{MR2801182,MR3590527,MR3663634} for our work.
	
	Secondly, following the strategy in \cite{MR2801182}, to study the uniqueness of our barycenters we introduce dual formulations of the Hellinger-Kantorovich barycenter problem $(\cP)$. To achieve the uniqueness of barycenters, we need that the term $\sum_{i=1}^p \int_X \lambda_if_id\mu$ in formula (\ref{F-crucial formula}) on page 20 does not depend on the choice of solutions $\mu$ of $(\cP)$. That is the reason we put the constraint $\sum_{i=1}^p\lambda_if_i=0$ into the dual problem $ (\cP^\ast)$ (page 13). And to show that the dual problem $ (\cP^\ast)$ is actually a dual formulation of problem $(\cP)$, we introduce another (normalized) dual problem $ (\cP_0^\ast)$ of $(\cP)$ to finish this task via convex analysis. In addition, to get our formula (\ref{F-crucial formula}) we need to show that problem $ (\cP^\ast)$ has solutions. To do this, we need that the limit function of a sequence of functions maximizing $ (\cP^\ast)$ still satisfies the constraints in $ (\cP^\ast)$. Therefore, we can not apply directly the duality formula of $HK$ in \cite[Theorem 7.21]{MR3763404} as the constraint $\sup f<1$ there is not closed under the pointwise limit. Instead, we introduce the sets of functions $F_i, i=1,\cdots p$ in definition \ref{D-$F_i$}, and establish variants of \cite[Theorem 7.21 (i) and (ii)]{MR3763404} in lemmas \ref{L-variants of duality formula of HK1}, \ref{L-formula of HK in term of S}, \ref{L-formula of HK and LET in terms of $F_i$} to overcome this obstacle. All these steps hold for general metric spaces $X$. Finally, applying duality results in \cite{MR3763404} for the Hellinger-Kantorovich distance in terms of Logarithmic Entropy Transport problems, and the uniqueness of optimal plans for these problems \cite[Theorem 6.6]{MR3763404}, we achieve the unique barycenter under some mild conditions for our starting measures when $X=\R^n$. Note that although the authors stated \cite[Theorem 6.6]{MR3763404} only for the Euclidean case $X=\R^n$, their proof still holds for all metric spaces for which we can apply Rademacher's theorem. For simplicity, we only present our result on the uniqueness of Hellinger-Kantorovich barycenters for the case $X=\R^n$. Furthermore, our dual formulation $ (\cP^\ast)$ of the original barycenter problem is useful for us to compute Hellinger-Kantorovich barycenters (example \ref{E-dual formulation of HK barycenters}).    
		
	Finally, similar to \cite{MR2801182,MR3004954} in Wasserstein spaces, we propose and study the following homogeneous multimarginal problem
	\[\inf\left\{ \int_{\mathfrak{C}^p}c((\eta_1,\ldots,\eta_p))d\boldsymbol{\alpha},\boldsymbol{\alpha}\in\cM_2(\mathfrak{C}^p),\mathfrak{h}^2_i\boldsymbol{\alpha}=\mu_i \right\},\]	
	where the definition of homogeneous marginals $\mathfrak{h}^2_i\boldsymbol{\alpha}$ is given at definition 1 below, and the cost function $c:\mathfrak{C}^p\to [0,\infty)$ is defined by $$c(\boldsymbol{\eta}):=
	\inf_{\eta\in\mathfrak{C}}\sum_{i=1}^p\lambda_id_\mathfrak{C}^2(\eta,\eta_i) \mbox{ for every } \boldsymbol{\eta}=(\eta_1,\dots,\eta_p)\in \mathfrak{C}^p.$$
	
	 We also establish relations of their minimum solutions with Hellinger-Kantorovich barycenters whenever $X$ is a Polish space having property $(BC)$ in theorem \ref{bary+multi}. As a consequence, in theorem \ref{T-consistency} we get a consistency result of Hellinger-Kantorovich barycenters which has been proved for Wasserstein spaces in \cite{MR3338645}. We also illustrate an explicit formula of Hellinger-Kantorovich barycenters via solutions of the homogeneous multimarginal problem when $X$ is the $(n-1)$-sphere $\mathbb{S}^{n-1}$.
	
	Our paper is organized as following. In section 2, we review Wasserstein distances, Hellinger-Kantorovich distances, and Logarithmic Entropy Transport. In section 3, we introduce metric spaces satisfying property $(BC)$ and prove the existence of Hellinger-Kantorovich barycenters for metric spaces having property $(BC)$. In section 4, we establish dual formulations of our original barycenter problem and prove the uniqueness of barycenters. And in the last section, we show our results for homogeneous multimarginal problems.     
	
	We also study barycenters for generalized Wasserstein spaces in our companion paper \cite{CT}.
	
	\textbf{Relation to \cite{FMS}.} After we posted our preprint on Arxiv in September 2019, soon after that in October 2019, Friesecke, Matthes and Schmitzer announced the related article \cite{FMS} which studies the barycenters for the Hellinger-Kantorovich distance over $\R^d$. Their setting in the article is restricted to compact, convex subsets of $\R^d$ while ours focuses on more general metric spaces beyond compact ones. However, their exposition is rather different from ours. First, the authors provide an alternate proof for the uniqueness of the Hellinger-Kantorovich barycenters over $\R^d$ without using their dual formulations. Second, using a detailed study of the convexity properties of the multimarginal cost, they establish a multimarginal formulation of the Hellinger-Kantorovich barycenter problem. The latter result relies heavily on convex duality between positively 1-homogeneous integral functionals on measures and indicator functions on continuous functions \cite{Rock}. In addition, they also provide a detailed analysis of the Hellinger-Kantorovich barycenters between Dirac measures and show corresponding numerical illustrations. In contrast, with a different method we introduce a homogeneous multimarginal formulation of the Hellinger-Kantorovich barycenter problem for more general metric spaces.       
	
	\textbf{Acknowledgements:} Part of this paper was carried out when N. P. Chung visited University of Science, Vietnam National University at Ho Chi Minh city in summer 2019. He thanks the math department there for its warm hospitality. The authors were partially supported by the National Research Foundation of Korea (NRF) grants funded by the Korea government No. NRF- 2016R1A5A1008055 , No. NRF-2016R1D1A1B03931922 and No. NRF-2019R1C1C1007107. We thank the anonymous referee for her/his useful comments which vastly improve the paper.      	
	\section{Preliminaries}
	Given a metric space $(X,d)$, we denote by $\cM(X)$ and $\cP(X)$ the spaces of all finite nonnegative Radon measures and all probability Radon measures on $X$, respectively. The set $\cP_2(X)$ (resp. $\cM_2(X))$ is defined as the set of all $\mu\in \cP(X)$ (resp. $\cM(X)$) such that there exists some (and therefore every) $x_0\in X$ such that $\int_Xd^2(x_0,x)d\mu(x)<\infty$. Given two measures $\mu_1,\mu_2\in\mathcal{P}_2(X)$, the Wasserstein distance between $\mu_1$ and $\mu_2$ is defined by $$W_2(\mu_1,\mu_2):=\left(\inf_{\pi\in \Pi (\mu_1,\mu_2)}\int_{X\times X}d^2(x,y)d\pi(x,y)\right)^{1/2},$$
	where $\Pi (\mu_1,\mu_2)$ is the set of all probability measures $\pi\in \cP(X\times X)$ such that its marginals are $\mu_1$ and $\mu_2$.
	
	For every $\mu\in \cM(X)$ and every measurable map $T:X\to Y$ between metric spaces $X$ and $Y$ we will denote by $T_\#\mu\in \cM(Y)$ the push forward measure defined by $T_\#\mu(A):=\mu(T^{-1}(A))$ for every Borel set $A$ in $Y$.
	
	Given $\mu,\nu\in \mathcal{M}(X)$, we say that $\mu$ is absolutely continuous with respect to $\nu$ and write $\mu \ll \nu$ if $\nu(A)=0$ yields $\mu(A)=0$ for every Borel subset $A$ of $X$. We call $\mu$ and $\nu$ mutually singular and write $\mu \perp \nu$ if there exists a Borel subset $B$ of $X$ such that $\mu(B)=\nu(X\backslash B)=0$. 	
	
	Let $(X,d)$ be a metric space. For every $a\geq 0$, we will denote by $d_a:=d\wedge a$ the truncated metric, i.e. $d_a(x,y)=\min\{d(x,y),a\}$ for every $x,y\in X$. The (Euclidean) cone $\fC$ of $X$, as a set, is the quotient space of $X\times [0,\infty)$ by the equivalence relation $\sim$ given by $(x_1,r_1)\sim (x_2,r_2)$ if and only if $r_1=r_2=0$ or $r_1=r_2,x_1=x_2$. We denote the equivalence classes by $\eta=[x,r]$. The special class $[x,0]$ is denoted by $\mathfrak{o}$. Then the function $d_\fC:\fC\times \fC\to [0,\infty)$ defined by
	$$d_\fC^2([x_1,r_1],[x_2,r_2]):=r_1^2+r_2^2-2r_1r_2\cos(d_\pi(x_1,x_2))$$
	is a metric. 
	
	Fixing a point $\bar{x}\in X$, we define two maps $\mathbf{r}:\fC\to [0,\infty), \mathbf{r}([x,r]):=r$ and $\mathbf{x}:\fC\to X, \mathbf{x}([x,r]):=x$ if $r>0$ and equals to $\bar{x}$ if $r=0$. For every $n\in \N$, we will denote by $\boldsymbol{\eta}=(\eta_1,\dots, \eta_n)=([x_1,r_1],\dots,[x_n,r_n])\in \fC^{ n}$, and define 
	$\mathbf{r}_i(\boldsymbol{\eta}):=\mathbf{r}(\eta_i)=r_i$, $\mathbf{x}_i(\boldsymbol{\eta}):=\mathbf{x}(\eta_i)$. 
\begin{definition}	
\label{D-homogeneous marginals}
	For every $\boldsymbol{\alpha}\in  \cM(\fC^{ n})$ with $\int_{\fC^{ n}}\sum_{i=1}^n\mathbf{r}_i^2d\boldsymbol{\alpha}<\infty$,
	we define its homogeneous marginals 
	$$\mathfrak{h}_i^2(\boldsymbol{\alpha}):=(\mathbf{x}_i)_{\#}(\mathbf{r}_i^2\boldsymbol{\alpha})\in \cM(X), i=1,\dots, n.$$ 	
	When $n=1$ we will write $\mathfrak{h}$ instead of $\mathfrak{h}_1$.
\end{definition}	
	
	Given $\mu_1,\mu_2\in \cM(X)$, the Hellinger-Kantorovich distance between $\mu_1$ and $\mu_2$ is defined by
	$$HK^2(\mu_1,\mu_2):=\inf\big\{\int_{\fC^2}d_\fC^2(\eta_1,\eta_2)d\boldsymbol{\alpha}:\boldsymbol{\alpha}\in \cM_2(\fC^2), \mathfrak{h}_i^2\boldsymbol{\alpha}=\mu_i, i=1,2\big\}.$$
	For every $R\geq 0$, let us set 
	$$\fC[R]:=\{[x,r]\in \fC:r\leq R\}.$$
	
	Let $\ell:[0,\infty]\to [0,\infty]$ be the function defined by $\ell(t):=\log(1+\tan^2(\min\{t,\pi/2\})),$ for every $t\in [0,\infty]$.
	For every $\gamma\in \cM(X\times X)$, its marginals are denoted by $\gamma_1$ and $\gamma_2$. Given $\mu_1,\mu_2\in \cM(X)$, we define 
	\begin{align*}
	\LET(\mu_1,\mu_2):=\inf_{\gamma\in \cM(X\times X)}\left\{\sum_{i}\int_X(\sigma_i\log\sigma_i-\sigma_i+1)d\mu_i+\int_{X\times X}\ell(d(x_1,x_2))d\gamma \right\}, 
	\end{align*}
	where $\sigma_i=\frac{d\gamma_i}{d\mu_i}$ is the Radon-Nikodym derivative of $\gamma_i$ with respect to $\mu_i$, and we also follow the convention that $0\log 0=0$. We denote by $\Opt_{\LET}(\mu_1,\mu_2)$ the set of all $\gamma\in \cM(X\times X)$ minimizing the above $\inf$.
	
	As we use the following results several times, we recall them.
	\begin{lem}
		\label{L-lifting gluing measure}
		(\cite[Remark 7.12]{MR3763404})
		Let $(X,d)$ be a metric space. For every $n\geq 2$ and every $\mu_1,\dots,\mu_n\in \cM(X)$, there exists $\boldsymbol{\beta}\in \cP_2(\fC^n)$ concentrated on $\{\boldsymbol{\eta}\in \fC^n:\sup_j\mathrm{r}_j(\boldsymbol{\eta})\leq \boldsymbol{\Xi}\}$ such that 
		$$\mathfrak{h}_j^2\boldsymbol{\beta}=\mu_j \mbox{ and } HK^2(\mu_1,\mu_j)=\int d^2_\fC(\eta_1,\eta_j)d\boldsymbol{\beta}, j=1,\dots, n,$$ where 
		$\boldsymbol{\Xi}=\sqrt{\mu_1(X)}+\sum_{j=2}^nHK(\mu_1,\mu_j)$.
	\end{lem}
	\begin{thm}
		\label{T-identity between HK and LET}
		(\cite[Theorem 6.3 and Theorem 7.20]{MR3763404})
		Let $(X,d)$ be a metric space. For all $\mu_1,\mu_2\in \cM(X)$ we have
		\begin{align*}
		HK^2(\mu_1&,\mu_2)=\LET(\mu_1,\mu_2)=\sup\big\{\sum_i\int_X f_id\mu_i:f_i\in LSC_s(X), \sup f_i<1,\\
		&(1-f_1(x_1))(1-f_2(x_2))\geq \cos^2(d_{\pi/2}(x_1,x_2)) \mbox{ for every } x_1,x_2\in X\big\},
		\end{align*} 
		where $LSC_s(X)$ is the space of all lower semi-continuous functions on $X$ with finite images.
		The identity still holds if we replace the space $LSC_s(X)$ by $C_b(X)$, the space of all bounded continuous functions on $X$. 
	\end{thm}
	
	\section{Existence of Hellinger-Kantorovich barycenters}
	%\section{Barycenter on Hallinger-Kantorovich space}
	Let $(X,d)$ be a metric space. Let $p\geq 2$ be an integer, $\lambda_1,\lambda_2,\ldots,\lambda_p $ be positive real numbers satisfying $ \sum_{i=1}^p\lambda_i=1 $ and $ \mu_1,\mu_2,\ldots,\mu_p \in\cM(X)$. We consider the following problem 
	$$ (\cP) \mbox{   }  \inf_{\mu\in \cM(X)}J(\mu)=\sum_{i=1}^p\lambda_i{HK}^2(\mu,\mu_i). $$
	
	We call a solution of problem $(\cP)$ a Hellinger-Kantorovich barycenter of the measures $\mu_i$ with weights $\lambda_i$. To study the existence of Hellinger-Kantorovich barycenters, we introduce some classes of metric spaces.	
\begin{definition}	
\label{D-(BC) property}
	We say that a metric space $(Y,d)$ has property $(B)$ if for every $k\in \N$, $y_1,\dots, y_k\in Y$, and every $\lambda_1,\dots,\lambda_k\geq 0$ with $\sum_{i=1}^k\lambda_i=1$, the problem $\inf_{y\in Y}\sum_{i=1}^k\lambda_id^2(y,y_i)$ attains the minimum and the set $\{z\in Y:\sum_{i=1}^k\lambda_id^2(z,y_i)=\min_{y\in Y} \sum_{i=1}^k\lambda_id^2(y,y_i)\}$ is also $\sigma$-compact. We say that a metric space $(Y,d)$ has property $(BC)$ if its cone $\fC$ has property $(B)$.
\end{definition}
	
	\begin{lem}
		\label{L-lower semi-continuity of cost functions induced from barycenters}
		Let $(Y,d)$ be a metric space. Then for every $k\in \N$, and every $\lambda_1,\dots,\lambda_k\geq 0$ with $\sum_{i=1}^k\lambda_i=1$, the map $c:Y^k\to [0,\infty)$ defined by $$c((y_1,\dots, y_k)):=\inf_{y\in Y}\sum_{i=1}^k\lambda_id^2(y,y_i) \mbox{ for every }	(y_1,\dots,y_k)\in Y^k	$$ is lower semi-continuous. 
	\end{lem}
	\begin{proof}
		Let $\{\boldsymbol{y}^n=(y_1^n,\dots,y_k^n)\}$ be a sequence in $Y^k$ converging to $\boldsymbol{y}=(y_1,\dots,y_k)$ in $Y^k$. Let $ y^n(j) $ be a sequence that minimizes $ \sum_{i=1}^k\lambda_id^2(y,y_i^n) $.
		As $ \{d(y^n(j),y_i)\}_{n,j} $ is bounded we get that for sufficient large $ n $, the difference between $ d^2(y^n(j),y^n_i)$ and $ d^2(y^n(j),y_i) $ is small for all $ j $. Hence the difference between $ \sum_{i=1}^k\lambda_id^2(y^n(j),y_i) $ and $ c(\boldsymbol{y}^n) $ is small for sufficient large $ n $ and sufficient large $ j $. Thus, $ \liminf_n c(\boldsymbol{y}^n)\ge c(\boldsymbol{y}) $.
	\end{proof}
	Now, we have versions of \cite[Lemma 7 and Theorem 8]{MR3663634} for metric spaces having property (B). 
	\begin{lem}
		\label{L-Borel barycenter map}
		Let $(Y,d)$ be a Polish space satisfying property $(B)$. Then for every $k\in \N$ and weights $\lambda_1,\dots, \lambda_k$, there exists a Borel map $T:Y^k\to Y$ associating $(y_1,\dots,y_k)$ to a minimum of $\inf_{y\in Y}\sum_{i=1}^k\lambda_id^2(y,y_i)$.
	\end{lem}
	\begin{proof}
		Applying \cite[Theorem 1]{BP} with $U=Y^k, V=Y$ to the subset 
		$$S=\big\{(y_1,\dots,y_k,y)\in U\times V: \sum_{i=1}^k\lambda_id^2(y,y_i)=\inf_{z\in Y}\sum_{i=1}^k\lambda_id^2(z,y_i)\big\},$$
		we get the result.
	\end{proof}
	Recall that given $\mu_1,\dots,\mu_k\in \cP(Y)$ we denote $\Pi(\mu_1,\dots, \mu_k)$ as the set of all $\alpha\in \cP(Y^k)$ with marginals are $\mu_1,\dots, \mu_k$. Using lemma \ref{L-lower semi-continuity of cost functions induced from barycenters}, lemma \ref{L-Borel barycenter map} and the same proof of \cite[Theorem 8]{MR3663634} we get
	\begin{thm}
		\label{T-existence of barycenter in Wasserstein space}
		Let $(Y,d)$ be a Polish space satisfying property $(B)$. Then for every $k\in \N$, $\mu_1,\dots, \mu_k\in \cP_2(Y)$ and weights $\lambda_1,\dots, \lambda_k$, there exists a measure $\alpha\in \Pi(\mu_1,\dots,\mu_k)$ minimizing 
		$$\beta \mapsto \int \inf_{y\in E}\sum_{i=1}^k\lambda_id^2(y_i,y)d\beta(y_1,\dots,y_k).$$
		Moreover, let $T:Y^k\to Y$ be a Borel map as in Lemma \ref{L-Borel barycenter map} then the measure $T_\#\alpha$ is a barycenter of $\mu_i$ with weight $\lambda_i$ in Wasserstein space and if this map $T$ is unique then $T_\#\alpha$ is the unique barycenter.	
		
	\end{thm}
	\begin{thm}
		Let $X$ be a Polish space having property (BC).	Let $ \mu_1,\mu_2,\ldots,\mu_p \in\cM(X)$ and let $ \lambda_1,\lambda_2,\ldots,\lambda_p $ be positive real numbers satisfying $ \sum_{i=1}^p\lambda_i=1 $. Then problem $(\cP)$ has solutions.
	\end{thm}
	\begin{proof}
		By theorem \ref{T-identity between HK and LET}, we get \begin{align*} \sum_{i=1}^p\lambda_i{HK}^2(\mu,\mu_i)=\sup\left\{\int_{X}\sum_{i=1}^{p}\lambda_if_id\mu+\sum_{i=1}^{p}\int_{X}\lambda_ig_id\mu_i\right\},\end{align*}
		where $f_i,g_i\in LSC_s(X)$ with $\sup f_i<1$, $\sup g_i<1$, and $(1-f_i(x))(1-g_i(y))\ge\cos^2(d_{\pi/2}(x,y))$ for every $x,y\in X$. Take a sequence $ \{\mu^n\} $ minimizing $ \inf_{\mu\in\cM(X)}J(\mu) $ then consider $ f_i=\frac{1}{2},g_i=-1 $ in the dual form to get $ J(\mu^n)+\sum_{i=1}^{p}\lambda_i\mu_i(X)\ge\frac{1}{2}\mu^n(X) $. Because $ J(\mu^n) $ is bounded and $ \mu_i(X) $ is finite for all $ i $ so $ \{ \mu^n \} $ is bounded.
		
		By lemma \ref{L-lifting gluing measure}, there are $ \alpha^n,\alpha^n_i\in\cP_2(\mathfrak{C}) $ with $ \alpha^n,\alpha^n_i $ being concentrated in $ \mathfrak{C}[\boldsymbol{\Xi}^n] $ where $ \boldsymbol{\Xi}^n=\sqrt{\mu^n(X)}+\sum_{i=1}^pHK(\mu^n,\mu_i) $ such that \begin{align*} \mathfrak{h}^2\alpha^n=\mu^n,\mathfrak{h}^2\alpha^n_i=\mu_i\text{ for } i=1,\ldots,p,\\HK(\mu^n,\mu_i)=W_{d_\mathfrak{C}}(\alpha^n,\alpha^n_i)\text{ for }i=1,\ldots,p. \end{align*}
		Note that $ \sum_{i=1}^pHK(\mu^n,\mu_i)\le \sum\dfrac{1}{4\lambda_i}+J(\mu^n) $ so $ \sum_{i=1}^pHK(\mu^n,\mu_i) $ is bounded. Thus, $ \{ \boldsymbol{\Xi}^n \} $ is bounded. Let $ \boldsymbol{\Xi} $ be an upper bound of $ \{ \boldsymbol{\Xi}^n \} $, we have that $ \alpha^n_i $ is concentrated in $ \mathfrak{C}[\boldsymbol{\Xi}] $ for all $ n $.
		
	 Recall that a sequence $\{\nu_n\}\subset \cM(X)$ is weak* convergent to $\nu\in \cM(X)$ if $\lim_{n\to\infty}\int_X\varphi d\nu_n=\int_X\varphi d\nu$, for every $\varphi\in C_b(X)$. By \cite[Lemma 7.3]{MR3763404} and Prokhorov's theorem $ \{\alpha^n_i\} $ is relatively compact in weak*-topology. Then there is a subsequence $ \{ \alpha^{n_k}_i \} $ converging weakly* in $ \cP_2(X) $. Let $ \alpha_i $ be the limit of $ \{ \alpha^{n_k}_i \} $. As $ \{ \alpha^{n_k}_i \} $ is concentrated on the bounded set $ \mathfrak{C}[\boldsymbol{\Xi}] $, we get 
		$$\int d_\fC^2([x,r],[x',0])d\alpha_i^{n_k}([x,r])\longrightarrow  \int d_\fC^2([x,r],[x',0])d\alpha_i([x,r]).$$ 
		Therefore from \cite[Theorem 6.9]{MR2459454}, $ W_{d_\mathfrak{C}}(\alpha^{n_k}_i,\alpha_i) $ converges to $ 0 $.
		
		Since $ \sum_{i=1}^pHK(\mu^n,\mu_i) $ is bounded then so is $ W_{d_\mathfrak{C}}(\alpha^{n_k},\alpha^{n_k}_i) $. Using the inequalities $\left|W_{d_\mathfrak{C}}(\alpha^{n_k}_i,\alpha^{n_k})-W_{d_\mathfrak{C}}(\alpha^{n_k},\alpha_i) \right|\leq W_{d_\mathfrak{C}}(\alpha^{n_k}_i,\alpha_i)$ and $W_{d_\mathfrak{C}}(\alpha^{n_k},\alpha_i)\leq W_{d_\mathfrak{C}}(\alpha^{n_k},\alpha^{n_k}_i)+ W_{d_\mathfrak{C}}(\alpha^{n_k}_i,\alpha_i)$, we have that
		\begin{align*}
		&\left|W^2_{d_\mathfrak{C}}(\alpha^{n_k},\alpha^{n_k}_i)-W^2_{d_\mathfrak{C}}(\alpha^{n_k},\alpha_i) \right|\\
		&=\left|W_{d_\mathfrak{C}}(\alpha^{n_k},\alpha^{n_k}_i)-W_{d_\mathfrak{C}}(\alpha^{n_k},\alpha_i) \right|\left|W_{d_\mathfrak{C}}(\alpha^{n_k},\alpha^{n_k}_i)+W_{d_\mathfrak{C}}(\alpha^{n_k},\alpha_i) \right|\\
		&\le W_{d_\mathfrak{C}}(\alpha^{n_k}_i,\alpha_i)\left( 2W_{d_\mathfrak{C}}(\alpha^{n_k},\alpha^{n_k}_i)+ W_{d_\mathfrak{C}}(\alpha^{n_k}_i,\alpha_i)\right) .
		\end{align*} 
		As $X$ is a Polish space having property $(BC)$, its cone $\fC$ is also Polish and has property $(B)$. From theorem \ref{T-existence of barycenter in Wasserstein space}, we get for every $ \varepsilon>0 $ there exists $ N>0 $ such that for every $ n_k>N $:\begin{align*} \varepsilon+J(\mu^{n_k})&\ge\sum_{i=1}^{p}\lambda_i W^2_{d_\mathfrak{C}}(\alpha^{n_k},\alpha_i)\\&\ge\min_{\bar{\alpha}\in\cP_2(\mathfrak{C})}\sum_{i=1}^{p}\lambda_i W^2_{d_\mathfrak{C}}(\bar{\alpha},\alpha_i)\\&=\sum_{i=1}^{p}\lambda_i W^2_{d_\mathfrak{C}}(\alpha,\alpha_i)\mbox{ }(\alpha \text{ is a minimizing solution})\\&\ge J(\mathfrak{h}^2\alpha).
		\end{align*}
		Hence, $ \inf_{\mu\in\cM(X)}J(\mu)\ge J(\mathfrak{h}^2\alpha) $ and this leads to the existence of solutions for the minimizing problem in Hellinger-Kantorovich space.
	\end{proof}
	
Now we illustrate examples of metric spaces satisfying property $(BC)$.
	\begin{example}	
		It is clear that every compact metric $X$ space has property $(B)$. 
		
		First, we show that every separable locally compact geodesic space $(X,d)$ also has property $(B)$. As $X$ is locally compact and geodesic, it is proper, i.e. every closed ball is compact, by Hopf-Rinow theorem. Hence for every $k\in \N$, $x_1,\dots, x_k\in X$, and every $\lambda_1,\dots,\lambda_k\geq 0$ with $\sum_{i=1}^k\lambda_i=1$, $\inf_{y\in X}\sum_{i=1}^k\lambda_id^2(y,x_i)$ always attains the minimum. Furthermore, because $X$ is locally compact and separable we get that the set $\{z\in X:\sum_{i=1}^k\lambda_id^2(z,x_i)=\min_{y\in X} \sum_{i=1}^k\lambda_id^2(y,x_i)\}$ is also $\sigma$-compact. Therefore, $X$ has property $(B)$.
		
		Now we will show that every compact metric space $X$ which is geodesic has property $(BC)$. From \cite[Lemma 7.1]{MR3763404} and \cite[Proposition I.5.10]{BH} we get that $\fC$ is locally compact and geodesic. On the other hand, we also have that $\fC$ is separable as $X$ is compact. Therefore, $\fC$ has property $(B)$ and hence $X$ has property $(BC)$.  
		
		Note that when $X$ is a compact metric space, from \cite[Corollary 7.16]{MR3763404} we know that the space $(\cM(X),HK)$ is proper and hence the existence of Hellinger-Kantorovich barycenters is straightforward. 
		
		Next we will illustrate that every complete $\CAT(1)$ space has property $(BC)$. Let us review $\CAT(\kappa)$ spaces, and more details can be found in \cite{BH}. Let $(Y,d)$ be a metric space and let $x,y\in Y$. A geodesic joining $x,y$ is a map $\varphi:[0,\ell]\to Y$ such that $\varphi(0)=x,\varphi(\ell)=y$, and $d(\varphi(s),\varphi(t))=|t-s|$ for every $s,t\in [0,\ell]$, where $\ell:=d(x,y)$. We say that $Y$ is a geodesic space if for every $x,y\in Y$ there exists a geodesic joining $x$ and $y$. Given $r>0$, we call that $Y$ is $r$-geodesic if for every $x,y\in Y$ with $d(x,y)<r$ there exists a geodesic joining $x$ and $y$.
		
		Given a real number $\kappa$, we denote model spaces by $M_\kappa$ the following metric spaces:
		\begin{enumerate}
			\item $M_0$ is the Euclidean space $\R^2$;
			\item if $\kappa>0$ then $M_\kappa$ is the sphere $\mathbb{S}^2$ with the metric is obtained by multiplying its distance function with $1/\sqrt{\kappa}$;
			\item if $\kappa<0$ then $M_\kappa$ is the hyperbolic plane $\mathbb{H}$ with the metric is obtained by multiplying its distance function with $1/\sqrt{-\kappa}$.	
		\end{enumerate}	
		We put $D_\kappa:=+\infty$ if $\kappa\leq 0$ and $D_\kappa:=\pi/\sqrt{\kappa}$ if $\kappa>0$. Let $(X,d)$ be a metric space. For $x,y\in X$, if there exists a geodesic joining we write $[x,y]$ to denote a definite geodesic segment joining $x$ to $y$. A geodesic triangle $\Delta$ in $X$ consists of three points $x,y,z\in X$ and a choice of three geodesic segments $[x,y],[y,z]$ and $[z,x]$. A point $p\in \Delta$ if $p$ is in $[x,y]\cup [y,z]\cup [z,x]$. A triangle $\overline{\Delta}(\bar{x},\bar{y},\bar{z})$ in $M_\kappa$ is called a comparison triangle for $\Delta([x,y],[y,z],[z,x])$ if $d(\overline{x},\overline{y})=d(x,y)$, $d(\overline{z},\overline{y})=d(z,y)$ and $d(\overline{x},\overline{z})=d(x,z)$. If $d(x,y)+d(y,z)+d(z,x)<2D_\kappa$, such a triangle always exists \cite[Lemma I.2.14]{BH}.
		
		Let $\Delta$ be a triangle geodesic in $X$ with perimeter less than $D_\kappa$ and let $\overline{\Delta}\subset M_\kappa$ be its comparison triangle. We say that $\Delta$ satisfies the $\CAT(\kappa)$ inequality if for every $p,q\in \Delta$ and all comparison points $\bar{p},\bar{q}\in\overline{\Delta}$, one has $d(p,q)\leq d(\bar{p},\bar{q})$.   
		
		If $\kappa\leq 0$ then $X$ is called $\CAT(\kappa)$ space if it is geodesic and every its geodesic triangle satisfies the $\CAT(\kappa)$ inequality. 
		If $\kappa>0$ then we $X$ is called $\CAT(\kappa)$ space if it is $D_\kappa$-geodesic and every its geodesic triangle with perimeter less than $2D_\kappa$ satisfies the $\CAT(\kappa)$ inequality. 
		
		We say that a metric space $(Y,d)$ is a (global) NPC space if it is complete, and for every $x_1,x_2\in Y$, there exists $y\in Y$ such that for all $z\in Y$, one has 
		$$d^2(z,y)\leq \frac{1}{2}d^2(z,x_1)+\frac{1}{2}d^2(z,x_2)-\frac{1}{4}d^2(x_1,x_2).$$
		From \cite[Theorem 1.3.2]{MR3241330}, we know that complete $\CAT(0)$ spaces consist of all NPC spaces. Applying \cite[Proposition 4.3]{Sturm} we get that every NPC space has property $(B)$. And as $X$ is $\CAT(1)$ if and only if its cone $\fC$ is $\CAT(0)$ \cite[Theorem II.3.14]{BH},  and $X$ is complete if and only if $\fC$ is complete \cite[Proposition I.5.9]{BH}, we obtain that $\CAT(1)$-spaces have property $(BC)$. Examples of $\CAT(1)$-spaces can be found in \cite[Section II.1]{BH}. 
	\end{example}
		
	Here are the reasons we put the condition $(BC)$ of metric spaces to prove the existence of Hellinger-Kantorovich barycenters. 	
	\begin{remark}
		\begin{enumerate}
			\item As $\fC$ is not locally compact unless $X$ is compact, we can not apply directly results of \cite{MR3590527,MR3663634} for the existence of Wasserstein barycenters on $\fC$.
			\item Hellinger-Kantorovich spaces in general are not NPC \cite[Theorem 8.8]{MR3763404}. Therefore, the existence of Hellinger-Kantorovich barycenters is not a consequence of \cite{Sturm}. Indeed, given two measures $\mu_1,\mu_2$ in $\R^n$ as in example \ref{E-nonuniqueness of HK barycenters}, there exists $t\in (0,1)$ such that $ \inf_{\mu\in \cM(\R^n)}(1-t)HK^2(\mu,\mu_1)+tHK^2(\mu,\mu_2)$ has at least two different minimizing solutions. Therefore, $(\cM(\R^n),HK)$ is not NPC by \cite[Proposition 4.3]{Sturm}. 
			\item If $(X,d)$ is a geodesic space then so is $(\cM(X),HK)$ \cite[Proposition 8.3]{MR3763404}. However, similar to the Wasserstein setting, if $X$ is not compact, $(\cM(X),HK)$ is not locally compact, in particular not proper; and the existence of barycenters is also less known for non-proper spaces. Here is our example inspired by \cite[Example 4, Example 9]{MR3542003} and  \cite[Remark 7.1.9]{Ambrosio} for this fact.  
		\end{enumerate}
		\begin{example}
			(The space $(\cM(X), HK)$ is locally compact only if $X$ is compact).
			
			Let $(X,d)$ be a metric space. First we calculate the distance $ HK(\mu_0,\mu_1) $ where $\mu_0=a_0\delta_{x_0},\mu_1=a_1\delta_{x_0}+b_1\delta_{x_1}$, $ a_0,a_1,b_1 $ are positive numbers, and $ x_0,x_1 $ are points in $ X $. 
			From \cite[Theorem 6.3 (b)]{MR3763404} we know that \[ \dfrac{a_1\sqrt{a_0}}{\sqrt{a_1+b_1\cos^2(d_{\pi/2}(x_0,x_1))}}\delta_{(x_0,x_0)}+\dfrac{b_1\cos^2(d_{\pi/2}(x_0,x_1))\sqrt{a_0}}{\sqrt{a_1+b_1\cos^2(d_{\pi/2}(x_0,x_1))}}\delta_{(x_0,x_1)}\in\Opt_{\LET}(\mu_0,\mu_1) \] and so using \cite[Theorem 6.3 (d)]{MR3763404} we get \[ HK^2(\mu_0,\mu_1)=\LET(\mu_0,\mu_1)=a_0+a_1+b_1-2\sqrt{a_0(a_1+b_1\cos^2(d_{\pi/2}(x_0,x_1)))}. \]
			If $(\cM(X),HK)$ is locally compact then we assume that there are $ \varepsilon>0 $ and a point $ x_0 $ such that the ball $ B'(\delta_{x_0},\varepsilon):=\left\{ \mu\in\cM(X):HK(\mu,\delta_{x_0})\le\varepsilon \right\} $ is compact. Consider any sequence $ (x_n)\in X $ then we have \begin{align*} HK^2(\delta_{x_0},a\delta_{x_0}+b\delta_{x_n})=&1+a+b-2\sqrt{a+b\cos^2(d_{\pi/2}(x_0,x_n))}\\=&\left( \sqrt{a+b\cos^2(d_{\pi/2}(x_0,x_n))}-1 \right)^2+b\sin^2(d_{\pi/2}(x_0,x_n)) \end{align*} for any positive numbers $ a,b $. Choose $ a:=1+\sqrt{2}\varepsilon $ and $ b:=\dfrac{\varepsilon^2}{2} $, we see that \begin{align*} HK^2(\delta_{x_0},a\delta_{x_0}+b\delta_{x_n})=&\left( \sqrt{a+b\cos^2(d_{\pi/2}(x_0,x_n))}-1 \right)^2+b\sin^2(d_{\pi/2}(x_0,x_n))\\\le&\left( \sqrt{1+\sqrt{2}\varepsilon+\dfrac{\varepsilon^2}{2}}-1 \right)^2+\dfrac{\varepsilon^2}{2}=\varepsilon^2. \end{align*} Therefore, $ (a\delta_{x_0}+b\delta_{x_n}) $ has a convergent subsequence in $HK$ distance. Applying \cite[Theorem 7.15 ]{MR3763404} we have that $(a\delta_{x_0}+b\delta_{x_n})$ has a convergent subsequence in the narrow topology on $\cM(X)$. Hence $ (x_n) $ must have a convergent subsequence in $X$ by \cite[Proposition 5.1.8 and Corollary 5.1.9]{Ambrosio}.
		\end{example}		
		
	\end{remark}	
	\section{Uniqueness and characterizations of Hellinger-Kantorovich barycenters}	
	In this section we will introduce dual formulations of the Hellinger-Kantorovich barycenter problem $(\cP)$, and investigate the uniqueness of barycenters. 
\subsection{Dual formulations of the Hellinger-Kantorovich barycenter problem}		
	Let $(X,d)$ be a metric space. For every $f\in C_b(X)$ with $\sup f\leq 1$, we define the function $Sf$ on $X$ by
	$$	Sf(y):=\left\{\begin{array}{l}\inf_{x\in X_f}\left\{1-\dfrac{\cos^2(d_{\pi/2}(x,y))}{1-f(x)}\right\} \mbox{ if } B_y(\pi/2)\subset X_f,\\-\infty \mbox{ otherwise },\end{array}\right.$$ where $ X_f=\left\{ x\in X|f(x)<1 \right\}. $
	
Our lemmas \ref{L-variants of duality formula of HK1}, \ref{L-formula of HK in term of S} and \ref{L-formula of HK and LET in terms of $F_i$} are variants of \cite[Theorem 7.21 (i) and (ii)]{MR3763404}. As we mentioned on the introduction, we need the constraint $f_i\in F_i$ in lemma \ref{L-formula of HK and LET in terms of $F_i$} instead of $\sup f_i< 1$ as in \cite[Theorem 7.21 (ii)]{MR3763404}, which is not closed under the pointwise limit, to show that the dual problem $(\cP^\ast)$ has solutions in proposition \ref{P-existence solutions of dual problems}. 
\begin{lem}
\label{L-variants of duality formula of HK1}	
For every $f\in C_b(X)$ with $\sup f\leq 1$, the function $Sf$ is upper semi-continuous, and in particular $Sf$ is measurable. Furthermore, if $ \sup f<1 $ then $ Sf $ is bounded and continuous with $ \sup Sf<1 $. 
\end{lem}	
	\begin{proof}
		Let $f\in C_b(X)$ with $ \sup f\le 1$, it's easy to see that $ \sup Sf<1 $. For $ y\in X $ if there exists $ x\in B_y(\pi/2):f(x)=1$ then for all $ y'\in B_y(\pi/2-d(x,y)):Sf(y')=-\infty $. Otherwise for $ y_n\in X, y_n\rightarrow y$ we have \begin{align*}
		\limsup_nSf(y_n)\le&\limsup_n\inf_{x\in X_f}\left\{1-\dfrac{\cos^2(d_{\pi/2}(x,y_n))}{1-f(x)}\right\}\\\le&\inf_{x\in X_f}\limsup_n\left\{1-\dfrac{\cos^2(d_{\pi/2}(x,y_n))}{1-f(x)}\right\}\\=&Sf(y).
		\end{align*}
		
		Let $f\in C_b(X)$ with $ \sup f=M<1$ then $ 1-\dfrac{\cos^2(d_{\pi/2}(x,y))}{1-f(x)}\ge 1-\dfrac{1}{1-M} $ for all $ x,y\in X $ which means $ Sf $ is bounded below. For $ y,y'\in X $ consider a sequence $ \{x_n\} $ that minimizes $\inf_{x\in X}\{ 1-\dfrac{\cos^2(d_{\pi/2}(x,y))}{1-f(x)}\} $. Then we get \begin{align*}
		& Sf(y')-Sf(y)\\&
		\le \liminf_n \dfrac{4}{1-M}
		\cos\left(\dfrac{1}{2}\left(d_{\pi/2}(x_n,y)+d_{\pi/2}(x_n,y')\right)\right)\cos\left(\dfrac{1}{2}\left(d_{\pi/2}(x_n,y)-d_{\pi/2}(x_n,y')\right)\right)\\
		&\sin\left(\dfrac{1}{2}\left(d_{\pi/2}(x_n,y)+d_{\pi/2}(x_n,y')\right)\right)\sin\left(\dfrac{1}{2}\left(d_{\pi/2}(x_n,y)-d_{\pi/2}(x_n,y')\right)\right)\\
		&\le \dfrac{2d_{\pi/2}(y,y')}{1-M}.
		\end{align*}
		Similarly, $ Sf(y)-Sf(y')\le \dfrac{2d_{\pi/2}(y,y')}{1-M}.$ Hence $ Sf $ is continuous.
		\end{proof}
	\begin{lem}
		\label{L-formula of HK in term of S}
		Let $(X,d)$ be a metric space and let $\mu_1,\mu_2\in \cM(X)$. Then 
		\[ HK^2(\mu_1,\mu_2)=\sup\left\{\int_{X}fd\mu_1+\int_{X}Sfd\mu_2,f\in C_b(X),\sup f<1\right\}. \]
	\end{lem}	
	\begin{proof}
		Given $f,g\in C_b(X)$ with $\sup f<1, \sup g<1$ and $ (1-f(x))(1-g(y))\ge\cos^2(d_{\pi/2}(x,y)) $ for every $x,y\in X$, we have $ g(y)\le 1-\dfrac{\cos^2(d_{\pi/2}(x,y))}{1-f(x)} $ for all $ x\in X $ which means $ g\leq Sf$. Therefore, from theorem \ref{T-identity between HK and LET}, we have 
		\[ HK^2(\mu_1,\mu_2)\leq\sup\left\{\int_{X}fd\mu_1+\int_{X}Sfd\mu_2,f\in C_b(X),\sup f<1\right\}. \]
		
		Now we prove that $ HK^2(\mu_1,\mu_2)\ge\int_{X}fd\mu_1+\int_{X}Sfd\mu_2 $ for every $f\in C_b(X)$ with $\sup f<1$. By \cite[Lemma 7.19]{MR3763404} there exists $ \boldsymbol\alpha\in\cM_2(\mathfrak{C}^2) $ such that $ \mathfrak{h}_1^2\boldsymbol{\alpha}=\mu_1,\mathfrak{h}_2^2\boldsymbol{\alpha}=\mu_2 $ and $ HK^2(\mu_1,\mu_2)=\int d_{\pi/2,\mathfrak{C}}^2d\boldsymbol{\alpha} $. Then we have
		\begin{align*}
		\int_{X}fd\mu_1+\int_{X}Sfd\mu_2=&\int_{\mathfrak{C}^2}\left( f(x_1)r_1^2+Sf(x_2)r_2^2 \right)d\boldsymbol\alpha\\
		=&\int_{\mathfrak{C}^2}\left( r_1^2+r_2^2-(1-f(x_1))r_1^2-(1-Sf(x_2))r_2^2 \right)d\boldsymbol\alpha\\
		\le&\int_{\mathfrak{C}^2}\left( r_1^2+r_2^2-2r_1r_2\cos(d_{\pi/2}(x_1,x_2)) \right)d\boldsymbol\alpha\\
		=&HK^2(\mu_1,\mu_2).
		\end{align*}
	\end{proof}
\begin{definition}
\label{D-$F_i$}	
	Given $p\geq 1$, $\mu_1,\dots, \mu_p\in \cM(X)$ and weights $\lambda_1,\dots, \lambda_p$, we define $F_i$ for every $i=1,\dots, p$ as the set of all functions $ f $ satisfying the following conditions:
	\begin{enumerate}
		\item $ f\in C_b(X) $ and $ f(x)\le 1,\forall x\in X $,
		\item If $ f(x)=1 $ then $ \mu_i(B'_x(\pi/8))=0 $, where $B_x'(s):=\{y\in X:d(x,y)\leq s\}$ for every $s>0$.
	\end{enumerate}	
\end{definition}	
	As $F_i$ contains all $f\in C_b(X)$ with $\sup f<1$, from lemma \ref{L-formula of HK in term of S} we get that $$HK^2(\mu,\mu_i)\leq \sup\left\{\int_{X}f_id\mu+\int_{X}Sf_id\mu_i,f_i\in F_i\right\}. $$
	And the proof in lemma \ref{L-formula of HK in term of S} of $HK^2(\mu,\mu_i)\geq \int_{X}f_id\mu+\int_{X}Sf_id\mu_i $ still holds for all $f_i\in F_i$. Therefore, we get the following lemma.
	\begin{lem}
		\label{L-formula of HK and LET in terms of $F_i$}
		Let $X$ be a metric space. Then for every $\mu\in \cM(X)$ we have \[ HK^2(\mu,\mu_i)=\sup\left\{\int_{X}f_id\mu+\int_{X}Sf_id\mu_i,f_i\in F_i\right\}. \]
	\end{lem}
	
	And hence $	\lambda_iHK^2(\mu,\mu_i)=\sup\left\{\int_{X}\lambda_if_id\mu+\int_{X}\lambda_iSf_id\mu_i,f_i\in F_i\right\},$ for every $\mu\in \cM(X)$.
	
	Now we consider a dual formulation of the Hellinger-Kantorovich barycenter problem $$ (\cP^*) \mbox{  } \sup\left\{F(f_1,\ldots,f_p)=\sum_{i=1}^{p}\int_X\lambda_iSf_id\mu_i:f_i\in F_i,\sum_{i=1}^{p}\lambda_if_i=0\right\}. $$	
	
	For every $i=1,\cdots, p$ and every $f\in C_b(X)$ with $\sup f<\lambda_i$ we define the function $S_{\lambda_i}f$ on $X$ by
	$$ S_{\lambda_i}f(y):=\inf_{x\in X}\left\{\lambda_i-\dfrac{\lambda_i^2\cos^2(d_{\pi/2}(x,y))}{\lambda_i-f_i(x)}\right\}  \mbox{ for every } y\in X.$$
	Then for every $ f\in C_b(X)$ with $\sup f<1 $ and every $i=1,\cdots, p$ we have $ S_{\lambda_i}(\lambda_if_i)=\lambda_iSf_i $.
	
	To prove $ \inf(\cP)=\sup(\cP^\ast) $ we investigate another maximum problem 
	\[ (\cP^*_0) \mbox{  }\sup\left\{F(f_1,\ldots,f_p)=\sum_{i=1}^{p}\int_XS_{\lambda_i}f_id\mu_i:f_i\in C_0(X),\sup f_i<\lambda_i,\sum_{i=1}^{p}f_i=0\right\}. \]
	For $ f\in C_0(X) $ we define \[ H_i(f):=\left\{\begin{array}{l}-\int_XS_{\lambda_i}fd\mu_i\text{ if }\sup f<\lambda_i\\+\infty\text{ otherwise}.\end{array}\right. \] Then we can define the Legendre-Fenchel transform of $H_i$ with respect to the dual pairing $ \left< f,\mu \right>:=\int_X fd\mu$ for every $f\in C_0(X),\mu\in \cM_s(X)$ as follows
	\begin{align*}
	H_i^\ast(\mu):&=\sup_{f\in C_0(X)}\left\{ \int_{X}fd\mu-H_i(f)\right\}\\
	&=\sup_{f\in C_0(X),\sup f<\lambda_i}\left\{ \int_{X}fd\mu+\int_{X}S_{\lambda_i}fd\mu_i \right\}, 
	\end{align*}	
	where $\cM_s(X)$ is the set of all Radon (not necessarily nonnegative) measures on $X$. 
	\begin{lem}
		\label{L-HK in terms of Legendre transforms}
		For every Radon measure with finite total variation $ \mu $ we have
		\[ H_i^\ast(\mu)=\left\{\begin{array}{l}
		\lambda_iHK^2(\mu,\mu_i)\text{ if } \mu\in\cM(X) \\+\infty\text{ otherwise. }
		\end{array} \right. \]
	\end{lem}
	\begin{proof}
		In case that $ \mu $ is not nonnegative, there exists $ f\in C_0(X),f\ge0 $ such that $ \int_X fd\mu<0 $. Instead of $ f $, we take $ -f $ which means there is a function $ f\in C_0(X),f\le 0 $ such that $ \int_X fd\mu>0 $. Note that $ S_{\lambda_i} $ is order-reversing so $ S_{\lambda_i}(tf)\ge0 $ for every $ t\ge0 $. Thus, we have \[ H_i^\ast(\mu)\ge\sup_{t\ge0}t\int_Xfd\mu=+\infty. \]
		
		Now, let's consider $ \mu\in\cM(X) $. It's clear that $ \lambda_iHK^2(\mu,\mu_i)\ge H_i^\ast(\mu) $, suppose $ \lambda_iHK^2(\mu,\mu_i)> H_i^\ast(\mu) $ then there exists $ f_0\in C_b(X),\sup f_0<1 $ such that \[ \int_Xf_0d\mu+\int_XSf_0d\mu_i>\sup_{f\in C_0(X),\sup f<1}\left\{ \int_{X}fd\mu+\int_{X}Sfd\mu_i \right\}. \]
		By the definition of Radon measures, for $ \varepsilon>0 $ there exists a compact subset $ K_\varepsilon $ of $ X $ such that $ \mu(X\setminus K_\varepsilon),\mu_i(X\setminus K_\varepsilon)<\varepsilon. $ Recall that $ d(x,A):=\inf_{y\in A}d(x,y) $ for $ x\in X $ and $ A $ is a subset of $ X. $ Define the subset $ K'_\varepsilon:=\{x\in X:d(x,K_\varepsilon)\le\pi/2 \} $. Define \[ f_\varepsilon(x)=\dfrac{f_0(x)}{e^{d(x,K'_\varepsilon)}} \]
		then we have that $ f_0 $ is equal to $ f_\varepsilon $ on $ K'_\varepsilon $ and $ Sf_0 $ is equal to $ Sf_\varepsilon $ on $ K_\varepsilon $. Since $ f_0 $ is bounded and $ \sup f_0<1 $ there exist $ N<0,M>0$ such that $N\le f_0(x)\le M<1,\forall x\in X $ and hence $ N,M:N\le f_\varepsilon(x)\le M<1,$ for every $x\in X$, and every $\varepsilon>0 $. Since $ S $ is order-reversing, $ 1-\dfrac{1}{1-M}\le Sf_\varepsilon(y)\le 1-\dfrac{1}{1-N}, \forall y\in X,\varepsilon>0 $. As a consequence,
		\begin{align*} \left| \int_X(f_0-f_\varepsilon) d\mu+\int_X(Sf_0-Sf_\varepsilon) d\mu_i \right|\le&\int_X\left|f_0-f_\varepsilon\right| d\mu+\int_X\left|Sf_0-Sf_\varepsilon\right| d\mu_i\\\le&\varepsilon\left( M-N+\dfrac{1}{1-M}-\dfrac{1}{1-N} \right),\forall\varepsilon>0.
		\end{align*}
		We get a contradiction so $ \lambda_iHK^2(\mu,\mu_i)= H_i^\ast(\mu). $
	\end{proof}
	\begin{proposition}
	\label{P-dual formulations of barycenters}
	Let $(X,d)$ be a metric space. Then we have
		\[ \inf(\cP)=\sup(\cP^\ast)=\sup(\cP^\ast_0). \]
	\end{proposition}
	\begin{proof}
		First, it is clear that $\sup(\cP^\ast)=\sup(\cP^\ast_0)$. 
		
		Second, we will prove that $\inf(\cP)\geq\sup(\cP^\ast)$. Let $\mu\in \cM(X)$, and $f_1,\dots,f_p\in C_0(X)$ with $\sup f_i<\lambda_i$ and $\sum_{i=1}^pf_i=0$. From lemma \ref{L-HK in terms of Legendre transforms}, we get 
		$$\lambda_i HK^2(\mu,\mu_i)\geq \int_X f_id\mu+\int_XS_{\lambda_i}f_id\mu_i \mbox{ for every } i=1,\dots, p.$$
		Summing over $i$ and using that $\sum_{i=1}^pf_i=0$ we have $$\sum_{i=1}^p \lambda_i HK(\mu,\mu_i)\geq \sum_{i=1}^p \int_XS_{\lambda_i}f_id\mu_i.$$ 
		
		Now we only need to prove that $ \inf(\cP)=\sup(\cP^\ast_0) $. We define $ H(f) $ for $ f\in C_0(X) $ as follow: \[ H(f):=\inf\left\{ \sum_{i=1}^{p}H_i(f_i),f_i\in C_0(X),\sum_{i=1}^{p}f_i=f \right\}. \]
		Considering the duality pairing $ \left< f,\mu \right>:=\int_Xfd\mu $ for every $f\in C_0(X),\mu\in \cM_s(X)$, we will prove the following properties:\begin{enumerate}
			\item $ H $ is convex and continuous on $ \left\{f\in C_0(X):\sup f<1\right\} $. Hence, $ H^{\ast\ast}(0)=H(0) $.
			\item $ H^\ast=\sum_{i=1}^{p}H_i^\ast. $
			\item $ \inf\sum_{i=1}^{p}H_i^\ast=-\left( \sum_{i=1}^{p}H_i^\ast \right)^\ast(0). $
		\end{enumerate}
		Applying these properties, we get that \[ \inf(\cP)=\inf\sum_{i=1}^{p}H_i^\ast=-\left( \sum_{i=1}^{p}H_i^\ast \right)^\ast(0)=-H^{\ast\ast}(0)=-H(0)=\sup(\cP_0^\ast). \]
		
		Now we check that those properties indeed hold. For $ f,g\in C_0(X) $ and $ \sup f,\sup g\le\lambda_i $ we have\begin{align*} tS_{\lambda_i}f(y)+(1-t)S_{\lambda_i}g(y)=& t\inf_x\left\{ \lambda_i-\dfrac{\lambda_i^2\cos^2(d_{\pi/2}(x,y))}{\lambda_i-f(x)} \right\}\\&+(1-t)\inf_x\left\{ \lambda_i-\dfrac{\lambda_i^2\cos^2(d_{\pi/2}(x,y))}{\lambda_i-g(x)} \right\}\\\le&\inf_x\left\{ \lambda_i-t\dfrac{\lambda_i^2\cos^2(d_{\pi/2}(x,y))}{\lambda_i-f(x)}-(1-t)\dfrac{\lambda_i^2\cos^2(d_{\pi/2}(x,y))}{\lambda_i-g(x)} \right\}\\\le&\inf_x\left\{ \lambda_i-\dfrac{\lambda_i^2\cos^2(d_{\pi/2}(x,y))}{\lambda_i-tf(x)-(1-t)g(x)} \right\}\\=&S_{\lambda_i}(tf+(1-f)g)
		\end{align*} so $ H_i(tf+(1-f)g)\le tH_i(f)+(1-t)H_i(g). $ If $ \sup f $ or $ \sup g $ is not less than $ \lambda_i $ then $ H(f) $ or $ H(g) $ is $ +\infty $ and we still have the inequality.
		Hence, $ H_i $ is convex and so is $ H $. For $ f\in C_0(X) $ and $ \sup f=M<1 $ take $ \varepsilon $ such that $ M+\varepsilon<1 $ then for $ g:\|g-f\|<\varepsilon $ take $ g_i=\lambda_ig $ we have \begin{align*}
		H_i(g_i)=&\int_X\sup_x\left\{ \dfrac{\lambda_i^2\cos^2(d_{\pi/2}(x,y))}{\lambda_i-g_i(x)}-\lambda_i \right\}d\mu_i\\\le&\lambda_i\mu_i(X)\left( \dfrac{1}{1-M-\varepsilon}-1 \right).
		\end{align*}
		Thus, $ H(g) $ is bounded above in $ B_f(\varepsilon) $ and is continuous by \cite[Lemma I.2.1]{MR1727362}. From \cite[Proposition I.4.1]{MR1727362} we get that $ H^{\ast\ast} $ is the $ \Gamma $-regularization of $ H $. The $ \Gamma $-regularization coincides with the lower-semicontinuous regularization (\cite[Proposition I.3.3]{MR1727362}) since $ H $ is convex and admits a continuous affine minorant. As $ H $ is continuous at 0, it also coincides with its lower-semicontinuous regularization at 0 (\cite[Corollary I.2.1]{MR1727362}). Hence $ H^{\ast\ast}(0)=H(0) $.
		
		Now we will prove the second property. Note that if $ \sup f\ge 1 $ and $ \sum_{i=1}^pf_i=f $ then there must exist $ j $ such that $ \sup f_j\ge\lambda_j $ which means there always exists $ j $ such that $ H_j(f_j)=+\infty $ so $ H(f)=+\infty $ when $ \sup f\ge 1 $. Therefore, we have that
		\begin{align*}
		&H^\ast(\mu)=\sup\left\{ \int_Xfd\mu-H(f),f\in C_0(X)\right\}\\
		=&\sup\left\{ \int_Xfd\mu-H(f),f\in C_0(X),\sup f<1 \right\}\\
		=&\sup\left\{ \int_Xfd\mu+\sup\{ -\sum_{i=1}^p H_i(f_i),f_i\in C_0(X),\sup f_i<\lambda_i,\sum_{i=1}^p f_i=f \},f\in C_0(X),\sup f<1 \right\}\\=&\sum_{i=1}^p\sup\left\{ \int_Xf_id\mu-H_i(f_i),f_i\in C_0(X),\sup f_i<\lambda_i \right\}=\sum_{i=1}^p H_i^\ast(\mu).
		\end{align*}
		
		The last one follows immediately from
		\begin{align*}
		-\left( \sum_{i=1}^{p}H_i^\ast \right)^\ast(0)=-\sup\left\{ \int_X0d\mu-\sum_{i=1}^p H_i^\ast(\mu) \right\}=\inf\sum_{i=1}^p H_i^\ast.
		\end{align*}
	\end{proof}
	
	Before proving the existence of solutions for problem $ (\cP^\ast) $ we recall several notions.
	
	The $c$-transform of a function $\varphi:X\to \R\cup\{+\infty\}$ for a cost function $c:X\times X\to [0,+\infty]$ is defined by
	$$\varphi^c(y):=\inf_{x\in X}(c(x,y)-\varphi(x)) \mbox{ for every } y\in X,$$ 
	with the convention that the subtraction is $+\infty$ whenever $c(x,y)=+\infty$ and $\varphi(x)=+\infty$.
	It is easy to check that $\varphi^{cc}\geq \varphi$. And we call that $\varphi$ is $c$-concave if $\varphi^{cc}=\varphi$.
	
	For a uniformly continuous map $f$ between metric spaces $(X_1,d_1)$ and $(X_2,d_2)$, its modulus of continuity is the function $\omega:[0,\infty)\to [0,\infty]$ defined by $$\omega(t):=\sup\{d_2(f(x),f(y)):d_1(x,y)\leq t\} \mbox{ for every } t\in [0,\infty).$$ It is clear that $\omega$ is nondecreasing, $\omega(0)=0$, $d_2(f(x),f(y))\leq \omega(d_1(x,y))$ for every $x,y\in X_1$, and $\omega$ is continuous at $0$. Furthermore if $f$ is bounded the image of $\omega$ is indeed in $[0,\infty)$.
	
	We also recall that a space is proper if every of its closed balls is compact.
	\begin{proposition}
	\label{P-existence solutions of dual problems}
		Let $(X,d)$ be a separable metric space which is proper. Then problem $ (\cP^\ast) $ has solutions.
	\end{proposition}
	\begin{proof}
		Because the supremum of problem $(\cP^\ast) $ stays the same if we strengthen the constraint to $ f_i\in C_b(X) $ and $ \sup f_i<1 $ so we take $ \{f_i^n\} $ to be a sequence satisfying that and maximizing $ (\cP^\ast) $. Take $ \varphi_i^n=-\log(1-f_i^n) $ then $ (\varphi_i^n)^c=-\log(1-Sf_i^n) $ with the cost function $ c(x,y):=-\log(\cos^2(d_{\pi/2}(x,y))). $ Problem $ (\cP^\ast) $ is equivalent to the problem:\[ \sup\left\{ \sum_{i=1}^{p}\int_X\lambda_i(1-e^{-\varphi_i^c})d\mu_i:\varphi_i\in \Phi_i,\sum_{i=1}^{p}\lambda_ie^{-\varphi_i}=1 \right\}, \] where $ \Phi_i $ is the set of all functions $ \varphi $ satisfying:\begin{enumerate}
			\item $ \varphi:X\to\R\cup\{ +\infty \} $ is continuous.
			\item If $ \varphi(x)=+\infty $ then $ \mu_i(B'_x(\pi/8))=0 $.
		\end{enumerate}
		
		For all $ i=1,\dots,p-1 $, instead of $ \varphi_i^n $ we can consider $ \tilde{\varphi_i^n}:=(\varphi_i^n)^{cc} $ and $ \tilde{\varphi_p^n}=-\log\left( \dfrac{1-\sum_{i=1}^{p-1}\lambda_ie^{-\tilde{\varphi_i^n}}}{\lambda_p} \right). $ This argument is valid since $ \tilde{\varphi_i^n}\ge \varphi_i^n$ and $ (\tilde{\varphi_i^n})^c= (\varphi_i^n)^c$, $ \tilde{\varphi_p^n}\le \varphi_p^n$ so $ (\tilde{\varphi_p^n})^c\ge (\varphi_p^n)^c $. Thus, we can always consider the case that $ \varphi_i^n $ is $ c- $concave for $ i=1,\ldots,p-1. $
		
		Let $ x_0 $ be a point in $ X $, we consider the compact set $ K_{x_0}=B'_{x_0}(\pi/8) $. For every $ x,y\in K_{x_0}$ with $d(x,y)\le\pi/4 $ we have $ e^{c(x,y)}\le 2 $ on $ K_{x_0} $. Because $ \{\varphi_i^n\} $ is a sequence that maximizes the problem, we can suppose there is a constant $ C $ such that \[ \sum_{i=1}^{p}\int_X\lambda_i(1-e^{-(\varphi_i^n)^c})d\mu_i\ge C \]
		and as a consequence \[ \sum_{i=1}^{p}\lambda_i\mu_i(X)-C\ge\sum_{i=1}^{p}\int_X\lambda_ie^{-(\varphi_i^n)^c}d\mu_i. \]
		As $ \varphi_i^n(x)\le c(x,y)-(\varphi_i^n)^c(y) $ for every $x,y\in X$, we have that for every $ x\in K_{x_0} $,
		\begin{align*} \sum_{i=1}^{p}\int_{K_{x_0}}\lambda_ie^{\varphi_i^n(x)}d\mu_i(y)\le&\sum_{i=1}^{p}\int_{K_{x_0}}\lambda_ie^{c(x,y)}e^{-(\varphi_i^n)^c(y)}d\mu_i(y)\\\le&2\left( \sum_{i=1}^{p}\lambda_i\mu_i(X)-C \right).
		\end{align*}
		If $ \mu_i(K_{x_0})>0 $ then $ \{\varphi_i^n \}$ is bounded from above in $ K_{x_0} $. On the other hand, $ \{\varphi_i^n \}$ is bounded below by $\log(\lambda_i)$ from the constraint $ \sum_{i=1}^{p}\lambda_ie^{-\varphi_i^n}=1 $. Thus, $ \{\varphi_i^n \}$ is bounded in $ K_{x_0} $.
		
		As $ \{\varphi_i^n \}$ is bounded below by $\log(\lambda_i)$ we get $ (\varphi_i^n)^c(y)=\inf_{x\in X}(c(x,y)-\varphi_i^n(x)) $ is bounded above by $ -\log(\lambda_i) $. Let $ M_i(K_{x_0}) $ be an upper bound for $\{ \varphi_i^n \}$ on $ K_{x_0} $. For $ x\in K_{x_0} $ and $ y\notin B'_x\left(t\right)$ we have $c(x,y)-(\varphi_i^n)^c(y)>M_i(K_{x_0}),$ where $t=\arccos\sqrt{\dfrac{\lambda_i}{e^{M_i(K_{x_0})}}}$. Therefore, for every $x\in K_{x_0}$ we get $$ (\varphi_i^n)^{cc}(x)=\inf_{y\in X}(c(x,y)-(\varphi_i^n)^c(y)) = \inf_{y\in B_x'(t)}(c(x,y)-(\varphi_i^n)^c(y)).$$
		
		Let $ x\in K_{x_0} $, let $ \delta:=\pi/2-t $ and take $ \delta'=\delta/4 $, the cost function $ c $ is continuous on $ (B'_x(\delta')\cap K_{x_0})\times B'_x\left(t+\delta'\right) $ so it is uniformly continuous and bounded.
		Let $ \omega:[0,\infty)\rightarrow[0,\infty) $ be its modulus of continuity then for every $ x'\in B'_x(\delta')\cap K_{x_0},y,y'\in B'_x\left(t+\delta'\right) $, we have \[ \left| c(x',y)-c(x,y) \right|\le\omega(d(x',x)). \]
		%\[ \left| c(x,y')-c(x,y) \right|\le\omega(d(y',y)). \]
		As  $ \varphi_i^n= (\varphi_i^n)^{cc}$ by its $ c $-concavity, we take a minimizing sequence $ \{y_k \}\subset B_x'(t+\delta')$ for $ \varphi_i^n(x) $ to get \[ \varphi_i^n(x')-\varphi_i^n(x)\le\liminf_k (c(x',y_k)-c(x,y_k))\le\omega(d(x',x)). \]
		As $B'_{x'}(t)\subset B'_x(t+\delta')$ and $$\varphi_i^n(x')=\inf_{y\in X}(c(x',y)-(\varphi_i^n)^c(y))=\inf_{y\in B'_{x'}(t)}(c(x',y)-(\varphi_i^n)^c(y)),$$ 	
		similarly as above, take a minimizing sequence $ \{y'_k\}\subset B_x'(t+\delta')$, we also get
		\[ \varphi_i^n(x)-\varphi_i^n(x')\le\liminf_k (c(x,y'_k)-c(x',y'_k))\le\omega(d(x,x')). \]
		Since $ \omega $ does not depend on $ n $, $\omega(0)=0$ and $\omega$ is continuous at 0, we get that 
		$ \varphi_i^n|_{K_{x_0}} $ is equi-continuous.
		
		By Ascoli-Arzel\`{a} theorem, if $ \mu_i(K_{x_0})>0 $ there exists a subsequence of $ \varphi_i^n|_{K_{x_0}} $(here still denoted by $ \varphi_i^n|_{K_{x_0}} $) and $ \varphi_{i,K_{x_0}}\in C(K_{x_0}) $ such that $ \varphi_i^n|_{K_{x_0}}\ $ uniformly converges to $ \varphi_{i,K_{x_0}},i=1,\ldots,p-1 $.
		
		As $ X $ is Polish, let $ A=\{x_1,x_2,\ldots \} $ be a countable dense subset of $ X $. Let $ K^1=\bigcup_{\mu_1(K_{x_j})>0}K_{x_j} $ and by a standard diagonal argument we get a subsequence $ \varphi_1^n $ that is pointwise convergent to $ \varphi_{1,K_{x_j}} $ whenever $ \mu_1(K_{x_j})>0 $, then we can define on $ K^1 $ \[ \varphi_1(x):=\varphi_{1,K_{x_j}}(x) \] when $ x\in K_{x_j}$ and $ \mu_1(K_{x_j})>0 $. For any point $ x_k\in A\setminus K^1 $, if $ \varphi_1^n(x_k) $ is bounded then we can take a subsequence of $ \varphi_1^n $ such that $ \varphi_1^n(x_k) $ is convergent and define $ \varphi_1(x_k) $ by its limit, otherwise there exists a subsequence of $ \varphi_1^n $ such that $ \varphi_1^n(x_k) $ increases to infinity and we define $\varphi_1(x_k)=+\infty$. Using a standard diagonal argument again and the density of $ A $, we construct a subsequence $ \varphi_1^n $ that pointwise converges to $ \varphi_1 $ and $ \varphi_1\in\Phi_1 $. With similar arguments, we can construct a subsequence $ (\varphi_1^n,\ldots,\varphi_{p-1}^n) $ that pointwise converges to $ (\varphi_1,\ldots,\varphi_{p-1}) $ and $ \varphi_i\in\Phi_i $ for every $i=1,\cdots,p-1$.
		
		We define \[ \varphi_p:=-\log\left( \dfrac{1-\sum_{i=1}^{p-1}\lambda_ie^{-\varphi_i}}{\lambda_p} \right) .\]
		If $ \mu_p(K_{x_j})>0 $, then \[ \sum_{i=1}^{p-1}\lambda_ie^{-\varphi_i^n}=1-\lambda_pe^{-\varphi_p^n}\le1-e^{-C_j}\] where $ C_j=\log\left(\dfrac{2\left( \sum_{i=1}^{p}\lambda_i\mu_i(X)-C \right)}{\lambda_p\mu_p(K_{x_j})}\right) $ so we also have $ \varphi_p^n $ pointwise converges to $ \varphi_p $ and $ \varphi_p\in\Phi_p $.
		
		Using Fatou's Lemma, we have that $ f_i=1-e^{-\varphi_i} $ solve problem $ (\cP^\ast). $
		To finish, we need to check that $ \limsup_nSf_i^n\le Sf_i $. For $ y\in X $, if $ B_y(\pi/2)\not\subset X_{f_i} $ then $ Sf_i(y)=-\infty $, we show that $ \limsup_n Sf_i^n(y)=-\infty $. Indeed, there exists $ x\in B_y(\pi/2) $ such that $ f_i(x)=1 $ so $ f^n_i(x) $ converges to 1. $ Sf_i^n(y)\le1-\dfrac{\cos^2(d_{\pi/2}(x,y))}{1-f_i^n(x)} $ which converges to $ -\infty $. If $ B_y(\pi/2)\subset X_{f_i} $ then \begin{align*}
		Sf_i(y)=\inf_{x\in X_{f_i}}\left\{ 1-\dfrac{\cos^2(d_{\pi/2}(x,y))}{1-f_i(x)}\right\}=&\inf_{x\in X_{f_i}}\limsup_n\left\{ 1-\dfrac{\cos^2(d_{\pi/2}(x,y))}{1-f_i^n(x)} \right\}\\\ge&\limsup_n\inf_{x\in X_{f_i}}\left\{ 1-\dfrac{\cos^2(d_{\pi/2}(x,y))}{1-f_i^n(x)} \right\}\\\ge&\limsup_n\inf_{x\in X}\left\{ 1-\dfrac{\cos^2(d_{\pi/2}(x,y))}{1-f_i^n(x)} \right\}\\=&\limsup_nSf_i^n.
		\end{align*}
	\end{proof}
\subsection{Uniqueness of Hellinger-Kantorovich barycenters}	
	Now we are ready to investigate the uniqueness of Hellinger-Kantorovich barycenters. Recall that given $\gamma,\mu\in \cM(X)$ with $\mu(X)+\gamma(X)>0$, from \cite[Lemma 2.3]{MR3763404} there exist Borel functions $\sigma,\rho:X\to [0,\infty)$ and a Borel partition $(A,A_\gamma,A_\mu)$ of $X$ satisfying the following properties:
	\[A=\{x\in X:\sigma(x)>0\}=\{x\in X:\rho(x)>0\}, \mbox{ } \sigma(x)\rho(x)=1, \forall x\in A;\]
	\[\gamma=\sigma\mu+\gamma^\perp, \gamma^\perp \perp \mu,\mu(A_\gamma)=\gamma^\perp(X\setminus A_\gamma)=0;\]
	\[\mu=\rho\gamma+\mu^\perp, \mu^\perp \perp \gamma,\gamma(A_\mu)=\mu^\perp(X\setminus A_\mu)=0;\]
	\[\sigma\in L^1(X,\mu), \rho\in L^1(X,\gamma).\]
	
	Furthermore, the sets $A,A_\gamma,A_\mu$ and the maps $\sigma,\rho$ are uniquely determined up to $(\mu+\gamma)$-negligible sets.
	
	Let $(X,d)$ be a metric space having property $(BC)$. Let $ \mu\in \cM(X) $ be a solution for $ (\cP) $ and $ (f_1,\dots, f_p)\in F_1\times \dots\times F_p $ with $\sum_{i=1}^p\lambda_if_i=0$ be a solution for $ (\cP^\ast) $. Then
\begin{align}	
\label{F-crucial formula}
	\sum_{i=1}^p\lambda_iHK^2(\mu,\mu_i)=\sum_{i=1}^p \int_X \lambda_iSf_id\mu_i=\sum_{i=1}^p \int_X \lambda_iSf_id\mu_i+\sum_{i=1}^p \int_X \lambda_if_id\mu.
	\end{align}
	From lemma \ref{L-formula of HK and LET in terms of $F_i$} we have 
	\[HK^2(\mu,\mu_i)= \int_X \lambda_iSf_id\mu_i+ \int_X \lambda_if_id\mu, \mbox{ for every } i=1,\dots, p,\]
	and then applying theorem \ref{T-identity between HK and LET} we get	$$\LET(\mu,\mu_i)=\int_X f_id\mu+\int_XSf_id\mu_i. $$
	
	Let $ \boldsymbol{\gamma}\in \Opt_{\LET}(\mu,\mu_i) $, then by \cite[Theorem 6.3 (b)]{MR3763404} we know $ \gamma_1,\gamma_2 $ are absolutely continuous with respect to $ \mu,\mu_i $ and there exist Borel sets $ A_j\subset\supp(\gamma_j) $ with $ \gamma_j(X\setminus A_j)=0 $ and Borel densities $ \sigma_1:A_1\to (0,\infty) $ of $\gamma_1$ w.r.t $\mu$, $ \sigma_2:A_2\to (0,\infty) $ of $\gamma_2$ w.r.t $\mu_i$ such that \begin{align*} 
	\sigma_1(x_1)\sigma_2(x_2)\ge&\cos^2(d_{\pi/2}(x_1,x_2))\quad\text{in }A_1\times A_2,\\
	\sigma_1(x_1)\sigma_2(x_2)=&\cos^2(d_{\pi/2}(x_1,x_2))\quad\boldsymbol{\gamma}\text{-a.e in }A_1\times A_2. 
	\end{align*}
	As $ \gamma_j(X\setminus A_j)=0 $ for $j=1,2$ we get $\gamma(A_1\times X)=\gamma(X\times X)=\gamma(X\times A_2)$ and hence $\gamma(A_1\times A_2)=\gamma(X\times X)$. Because $\sigma_1(x)\sigma_2(y)=\cos^2(d_{\pi/2}(x,y)$ for $ \gamma $-a.e $(x,y)$ in $ A_1\times A_2 $ and $ \sigma_j>0 $ in $ A_j $ we get that $ \gamma $ is concentrated on $\left\{ (x,y)\in A_1\times A_2|d(x,y)<\pi/2 \right\}$. Now we will show that
	\begin{align*}
	\dfrac{1-f_i(x)}{\sigma_1(x)}+\dfrac{1-Sf_i(y)}{\sigma_2(y)}\geq 2  \mbox{ for } \gamma -\mbox{a.e } (x,y) \mbox{ in } A_1\times A_2 .
	 %(1-f_i(x))(1-Sf_i(y))\ge \sigma_1(x)\sigma_2(y)  \mbox{ for } \gamma -\mbox{a.e } (x,y) \mbox{ in } A_1\times A_2 .
	 \end{align*}
	 Recall that the function $Sf_i:X\to \R\cup\{-\infty\}$ is defined by
	$$	Sf_i(y):=\left\{\begin{array}{l}\inf_{x\in X_{f_i}}\left\{1-\dfrac{\cos^2(d_{\pi/2}(x,y))}{1-f_i(x)}\right\} \mbox{ if } B_y(\pi/2)\subset X_{f_i},\\-\infty \mbox{ otherwise },\end{array}\right.$$ where $ X_{f_i}=\left\{ x\in X|f_i(x)<1 \right\}. $
	
	If $Sf_i(y)=-\infty$ then it is clear that $\dfrac{1-f_i(x)}{\sigma_1(x)}+\dfrac{1-Sf_i(y)}{\sigma_2(y)}\geq 2$. Therefore, we only consider $Sf_i(y)\ne -\infty$. In this case, by Cauchy-Schwarz inequality we have \[ \dfrac{1-f_i(x)}{\sigma_1(x)}+\dfrac{1-Sf_i(y)}{\sigma_2(y)}\ge2\sqrt{\dfrac{1-f_i(x)}{\sigma_1(x)}\dfrac{1-Sf_i(y)}{\sigma_2(y)}}. \] 
	 Hence, we only need to check $(1-f_i(x))(1-Sf_i(y))\ge \sigma_1(x)\sigma_2(y)$ for every $(x,y)\in A_1\times A_2$ with $ Sf_i(y)\ne-\infty $ and $ d(x,y)<\pi/2 $. Let $(x,y)\in A_1\times A_2$ satisfying $ Sf_i(y)\ne-\infty $ and $ d(x,y)<\pi/2 $.	Because $ Sf_i(y)\ne-\infty $, from the definition of $Sf_i$ we have $B_y(\pi/2)\subset X_{f_i}$ and hence $x\in X_{f_i} $. As $ Sf_i(y)=\inf_{z\in X_{f_i}}\left\{ 1-\dfrac{\cos^2(d_{\pi/2}(z,y))}{1-f_i(z)} \right\} $ we get that $$ (1-f_i(x))(1-Sf_i(y))\ge\cos^2(d_{\pi/2}(x,y)) = \sigma_1(x)\sigma_2(y) .$$ 
	And hence we finish the proof that \[ \dfrac{1-f_i(x)}{\sigma_1(x)}+\dfrac{1-Sf_i(y)}{\sigma_2(y)}\geq 2  \mbox{ for } \gamma -\mbox{a.e } (x,y) \mbox{ in } A_1\times A_2 . \]
	It implies that	
	\begin{align*}
    \LET(\mu,\mu_i)=&\int_X f_id\mu+\int_XSf_id\mu_i\\
	=&\int_Xd\mu-\int_X(1-f_i)d\mu+\int_Xd\mu_i-\int_X(1-Sf_i)d\mu_i\\
	=&\mu(X)+\mu_i(X)-\int_{A_1} \dfrac{1-f_i}{\sigma_1}d\gamma_1-\int_{A_2}\dfrac{1-Sf_i}{\sigma_2}d\gamma_2\\
	&-\int_{X\setminus A_1} (1-f_i)d\mu-\int_{X\setminus A_2} (1-Sf_i)d\mu_i\\
	\le&\mu(X)+\mu_i(X)-2\boldsymbol{\gamma}(A_1\times A_2)\\
	=&\mu(X)+\mu_i(X)-2\boldsymbol{\gamma}(X\times X).
	\end{align*}
	On the other we have $\mu(X)+\mu_i(X)-2\boldsymbol{\gamma}(X^2)= \LET(\mu,\mu_i)$ \cite[Theorem 6.3 (d)]{MR3763404}.
	Hence, $ \sigma_1=1-f_i $ $ \gamma_1 $-a.e on $ A_1 $,$ 1-f_i=0 $ $ \mu $-a.e on $ X\setminus A_1 $ and $ \sigma_2=1-Sf_i $ $ \gamma_2 $-a.e on $ A_2 $, $ 1-Sf_i=0 $ $ \mu_i $-a.e on $ X\setminus A_2 $. Because $ \sup Sf_i<1 $ we can consider $ \supp\mu_i=A_2 $. 
	
	We say that $\nu\in \cM(X)$ satisfies condition (*) if $d(x,\supp\nu)<\pi/2$ for every $x\in X$. 
	
	Because $ d $ must be greater than $ \pi/2 $ in $ \supp \mu^{\perp}\times \supp\mu_i $ by \cite[Theorem 6.3 (b)]{MR3763404} we get that if $\mu_i$ satisfies condition (*) then $ \supp\mu^{\perp}=\emptyset $ and hence $ \mu^{\perp} $ is the null measure.
	
	Now we assume that $ X =\R^n$ with the Euclidean distance for some $n\in \N$, and there exists some $ i $ such that $ \mu_i $ is absolutely continuous with respect to the Lebesgue measure. We will use the notion of approximate differential denoted by $\tilde{D}$ \cite[Definition 5.5.1]{Ambrosio}. Then by \cite[Theorem 6.6]{MR3763404} $ \sigma_2 $ is approximately differentiable $ \mu_i $-a.e. and the optimal plan $ \boldsymbol{\gamma} $ of $ \mu,\mu_i $ is uniquely determined by the push forward measure $ (\boldsymbol{t},Id)_{\#}\gamma_2 $, where \begin{align*}
	\boldsymbol{t}(x_2):=&x_2-\dfrac{\arctan(\| \boldsymbol{\xi}(x_2) \|)}{\| \boldsymbol{\xi}(x_2) \|}\boldsymbol{\xi}(x_2),\\
	\boldsymbol{\xi}(x_2):=&-\dfrac{1}{2}\tilde{D}\log\sigma_2(x_2).
	\end{align*}
	Since $ Sf_i $ is already determined so $ \sigma_2,\gamma_2,\gamma_1,\sigma_1 $ will be determined regardless of what $ \mu $ is. We can write $ \mu $ in the form $ \sigma_1^{-1}\gamma_1+\mu^{\perp} $. And if $ \mu_i $ satisfies condition (*) then $ \mu $ is uniquely determined. We conclude the result into the following proposition.
	\begin{thm}
		\label{T-uniqueness}
		Let $ \mu_1,\ldots,\mu_p $ be nonnegative Radon measures on $\R^n$ and let $ \lambda_1,\ldots,\lambda_p $ be positive real numbers satisfying $ \sum_{i=1}^{p}\lambda_i=1 $. Let $\mu$ be a solution for $(\cP)$ and $ (f_1,\dots, f_p)$ be a solution for $ (\cP^\ast )$. Taking $ A_i:=\supp\mu_i,\sigma_i:=1-Sf_i $ on $ A_i $,$ \gamma_i:=\sigma_i\mu_i $. 
		\begin{enumerate}
			\item  If there exists $i$ such that $ \sigma_i $ is $ \mu_i $-a.e. approximately differentiable on $ A_i $ (for instance $ \mu_i $ is absolutely continuous to Lebesgue measure), 
			then $ \mu $ must have the form $ \tilde{\sigma_i}^{-1}\tilde{\gamma_i}+\mu^{\perp} $, where $ \tilde{\sigma_i}:=1-f_i $ on $ \tilde{A_i}:=\boldsymbol{t}_i(A_i)\cap\{x\in\R^n|f_i(x)<1\} $ and $ \tilde{\gamma_i}=\boldsymbol{t}_{i\#}\gamma_i $ with 
			\begin{align*}
			\boldsymbol{t}_i(x):=&x-\dfrac{\arctan(\| \boldsymbol{\xi}_i(x) \|)}{\| \boldsymbol{\xi}_i(x) \|}\boldsymbol{\xi}_i(x),\\
			\boldsymbol{\xi}_i(x):=&-\dfrac{1}{2}\tilde{D}\log\sigma_i(x).
			\end{align*}
			% ($ d\ge\pi/2$ on $\supp\mu^{\perp}\times\supp\mu_i $ and $ 1-f_i=0 $ $ \mu^{\perp} $-a.e.)
			\item Furthermore, if $ \mu_i $ above also satisfies condition (*) then the solution $ \mu $ is unique.
		\end{enumerate}
	\end{thm}
	\begin{remark}
		Theorem \ref{T-uniqueness} would hold for all metric spaces having property $(BC)$ and for which \cite[Theorem 6.6]{MR3763404} is still valid. From the proof of \cite[Theorem 6.6]{MR3763404}, we see that it could be extended for all spaces that we can apply Rademacher's theorem, i.e. every Lipschitz function is approximately differentiable. For simplicity, we only present theorem \ref{T-uniqueness} for $\R^n$ only.   
	\end{remark}
	\begin{example}
	\label{E-dual formulation of HK barycenters}
		Let $ X=\R,p=2,\lambda_1=\lambda_2=1/2 $ and $ \mu_1=\mathcal{L}_{[-1-\pi,-\pi]}, \mu_2=\mathcal{L}_{[\pi,\pi+1]} $ where $ \mathcal{L} $ is the Lebesgue measure and $ \mathcal{L}_A $ is the restriction measure on Borel set $ A $. We define \[ f_1(x)= \left\{\begin{array}{cc}
		-1&\text{ if }x\le-\pi/2\\
		\dfrac{2x}{\pi}&\text{ if }-\pi/2\le x\le\pi/2\\
		1&\text{ if }x\ge\pi/2
		\end{array}\right. \]
		and $ f_2=-f_1 $. Now we check that $ (f_1,f_2) $ is a solution for $ (\cP^\ast) $. Let $(g_1,g_2)\in F_1\times F_2$ such that $\lambda_1g_1+\lambda_2g_2=0$. Then $g_1+g_2=0$, $ g_i\le1$, and hence $ -1\le g_i\le 1$. As $g_1\in F_1$ and $ \mu_1=\mathcal{L}_{[-1-\pi,-\pi]}$, by the definition of $F_1$ we must have $g_1<1$ on $[-1-\pi,-\pi]$. Because $ g_1 $ is bounded below by $ -1 $, we have $ Sg_1(y)\leq \inf_{x\in X_{g_1}}\left\{1-\dfrac{\cos^2(d_{\pi/2}(x,y))}{1-g_1(x)}\right\}\leq 1-\frac{1}{2}=\frac{1}{2}$ for every $y\in [-1-\pi,-\pi]$. Similarly, we have 
$ Sg_2\leq \frac{1}{2} $ on $ [\pi,1+\pi]$. By the definitions of $f_1,f_2$, we have that $(f_1,f_2)\in F_1\times F_2$, $\lambda_1f_1+\lambda_2f_2=0$, $ Sf_1=\frac{1}{2} $ on $ [-1-\pi,-\pi] $, and $ Sf_2=\frac{1}{2} $ on $ [\pi,1+\pi]$. As $ Sf_1,Sf_2 $ reach the largest possible values on the supports of $ \mu_1,\mu_2 $, respectively, we get that $ (f_1,f_2) $ is a solution of $ (\cP^\ast). $		
		
Applying proposition \ref{P-dual formulations of barycenters} we get that 
$$ \inf_{\mu\in \cM(X)}\lambda_1HK^2(\mu,\mu_1)+\lambda_2HK^2(\mu,\mu_2)=\frac{1}{2}\sum_{i=1}^2\int_XSf_id\mu_i=\frac{1}{2}.$$
 Consider $ \mu_1 $, as in theorem \ref{T-uniqueness}, we have $ A_1=[-1-\pi,-\pi],\sigma_1=1/2,\boldsymbol{t}_1=Id $ and hence a barycenter $ \mu $ has the form $ \dfrac{1}{4}\mu_1+\mu^{\perp} $ ( $ \supp\mu^{\perp}\subset[\pi/2,+\infty) $ ). Similarly consider $ \mu_2 $, we will get $ \mu=\dfrac{1}{4}\mu_1+\dfrac{1}{4}\mu_2 $. We can check this $ \mu $ is truly a barycenter. Indeed $ \gamma_i=1/2(Id,Id)_\#\mu_i\in\Opt_{LET}(\mu,\mu_i) $ (\cite[Theorem 6.3 (b)]{MR3763404}) so $ HK^2(\mu,\mu_i)=1/2 $.

In this case we still have the uniqueness of barycenter although both $\mu_1,\mu_2$ do not satisfy condition (*). However, in general we may not have it. 					
	\end{example}
	\begin{example}(\cite[Example 11 ii)]{MR3542003})
		\label{E-nonuniqueness of HK barycenters}
		Let $ X=\R^n,p=2 $ and $ \mu_1=a_1\delta_{x_1}, \mu_2=a_2\delta_{x_2} $ where $ a_i $ is a positive number and $ \delta_{x_i} $ is the Dirac measure at point $ x_i\in X $ such that $ |x_1-x_2|=\pi/2 $. We have two geodesics which connect $ \mu_1 $ and $ \mu_2 $. The first geodesic is the Hellinger curve $ \mu^H(s):=(1-s)^2a_1\delta_{x_1}+s^2a_2\delta_{x_2} $ and the second one is defined by $ \mu(s):=a(s)\delta_{x(s)} $ where $ a(s):=(1-s)^2a_1+s^2a_2,x(s):=(1-p(s))x_1+p(s)x_2 $ and $ p(s):=\left\{\begin{array}{l}\dfrac{2}{\pi}\arctan\left(\dfrac{s\sqrt{a_2}}{(1-s)\sqrt{a_1}}\right)\mbox{ if }s\in[0,1)\\1\mbox{ if }s=1. \end{array}\right.$
		
		For $ \mu_1,\mu_2 $ with different constant-speed geodesics $ \mu^1(s),\mu^2(s) $ which connect them, there is $ t\in(0,1) $ such that $ \mu^1(t)\ne\mu^2(t) $. Both $ \mu^1(t),\mu^2(t) $ minimize the problem $ \inf_{\mu\in \cM(X)}(1-t)HK^2(\mu,\mu_1)+tHK^2(\mu,\mu_2). $
	\end{example}
	\section{Homogeneous multimarginal problems and Hellinger-Kantorovich barycenters}
	Let $ (X,d) $ be a metric space and $ \mathfrak{C} $ be its cone, $ \mu_1,\ldots,\mu_p $ be nonnegative Radon measures on $ X $, $ \lambda_1,\ldots,\lambda_p $ be positive real numbers satisfying $\sum_{i=1}^p\lambda_i=1$. We define the cost function $c:\fC^p\to [0,\infty)$ by 
\begin{align*}	
	c(\boldsymbol{\eta}):=
	\inf_{\eta\in\mathfrak{C}}\sum_{i=1}^p\lambda_id_\mathfrak{C}^2(\eta,\eta_i) \mbox{ for every } \boldsymbol{\eta}=(\eta_1,\dots,\eta_p)\in \fC^p.
\end{align*}	
	As $c$ is the infimum of continuous functions we get that $c$ is upper semi-continuous and hence it is measurable.
	
	The multi-marginal problem for the cost function $c$ above has been studied first by Gangbo and \'{S}wi\c{e}ch for $\R^n$ \cite{GS}. Their result has been generalized in several directions in \cite{Pass2011,KimPass2014,MR3372314,MR3004954}. It also has been used to study Wasserstein barycenters \cite{MR2801182,MR3004954}. 
	
	Similar to \cite{MR2801182,MR3004954}, now we investigate the following homogeneous multimarginal problem and its connection with Hellinger-Kantorovich barycenters
	\begin{align*}
	\inf\left\{ \int_{\mathfrak{C}^p}c((\eta_1,\ldots,\eta_p))d\boldsymbol{\alpha},\boldsymbol{\alpha}\in\cM_2(\mathfrak{C}^p),\mathfrak{h}^2_i\boldsymbol{\alpha}=\mu_i \right\}.
	\end{align*}
	
	Before proving the existence of solutions for the homogeneous multimarginal problem, let us recall dilations. If $\boldsymbol{\alpha}\in \cM(\fC^p)$ and $\vartheta:\fC^p\to (0,\infty)$ is a Borel map in $L^2(\fC^p,\boldsymbol{\alpha})$ we define the map $prd_\vartheta:\fC^p\to \fC^p$ by
	$$prd_\vartheta(\boldsymbol{\eta})_i:=\eta_i\cdot (\vartheta(\boldsymbol{\eta}))^{-1}, \mbox{ for every } 
	\boldsymbol{\eta}=(\eta_1,\dots,\eta_p)\in \fC^p,i=1,\dots,p.$$
	\begin{lem}
		Let $(X,d)$ be a metric space. Then the function $c$ is continuous and the homogeneous multimarginal problem has a solution.
	\end{lem}
	\begin{proof}
		Applying lemma \ref{L-lower semi-continuity of cost functions induced from barycenters} we get that $c$ is lower semi-continuous and hence it is continuous.

		As there always exists $ \boldsymbol{\alpha}\in\cM_2(\mathfrak{C}^p),\mathfrak{h}^2_i\boldsymbol{\alpha}=\mu_i $ by \cite[(5.29)]{MR3763404}), we can take a minimizing sequence $ \boldsymbol{\alpha_n} $ of the homogeneous multimarginal problem. From the definition of $\mathfrak{h}_i$, we can restrict the 
		constraint in the problem to all $ \boldsymbol{\alpha}\in \cM_2(\fC^p)$ concentrated on $ \mathfrak{C}^p\setminus\{(\mathfrak{o},\ldots,\mathfrak{o})\} $, where $\mathfrak{o}$ is the point $[x,0]$ in $\fC$. Then using dilations  we can restrict furthermore the constraint to $ \boldsymbol{\alpha}\in\cP_2(\mathfrak{C}^p)$, with $\mathfrak{h}^2_i\boldsymbol{\alpha}=\mu_i $ and $ \boldsymbol{\alpha} $ is concentrated on $ \mathfrak{C}[\boldsymbol{\Xi}] $, where $ \boldsymbol{\Xi}=\sqrt{\sum_{i=1}^p\mu_i(X)} $. This is valid since $ c(prd_\vartheta(\boldsymbol{\eta}))=c(\boldsymbol{\eta})(\vartheta(\boldsymbol{\eta}))^{-2} $ for every Borel map $\vartheta:\fC^p\to (0,\infty)$, $\eta\in \fC^p$.
		Indeed, let $ \boldsymbol{\eta}^n $ and $ \overline{\boldsymbol{\eta}}^n $ be minimizing sequences for the barycenter problem for the point $ \boldsymbol{\eta} $ and the point $ prd_\vartheta(\boldsymbol{\eta}) $, respectively then we have that
		 \begin{align*} c(prd_\vartheta(\boldsymbol{\eta}))=&\lim_n\sum_{i=1}^p\lambda_id_\mathfrak{C}^2(\overline{\boldsymbol{\eta}}^n,\eta_i\cdot(\vartheta(\boldsymbol{\eta}))^{-1})\\
		=&\lim_n\sum_{i=1}^p\lambda_id_\mathfrak{C}^2(\overline{\boldsymbol{\eta}}^n\cdot\vartheta(\boldsymbol{\eta}),\eta_i)(\vartheta(\boldsymbol{\eta}))^{-2}\\
		\ge&c(\boldsymbol{\eta})(\vartheta(\boldsymbol{\eta}))^{-2}\\
		\\=&\lim_n\sum_{i=1}^p\lambda_id_\mathfrak{C}^2(\boldsymbol{\eta}^n\cdot(\vartheta(\boldsymbol{\eta}))^{-1},\eta_i\cdot(\vartheta(\boldsymbol{\eta}))^{-1})\\
		\ge&c(prd_\vartheta(\boldsymbol{\eta})). \end{align*}
		
		With such an estimate, the sequence $ \boldsymbol{\alpha_n} $ is bounded and equally tight  by \cite[lemma 7.3]{MR3763404}. Thus, it's relative compact in the weak* topology and each limit point still satisfies the constraint. As $c$ is continuous we get the result.
	\end{proof}
	
	Now, analogous to the results in \cite{MR2801182, MR3004954} for the Wasserstein setting, we establish relations between Hellinger-Kantorovich barycenters and solutions of the homogeneous multimarginal problem.	
	\begin{thm}
		\label{bary+multi}
		Let $(X,d)$ be a Polish metric space having property $(BC)$. Let $T:\fC^p\to \fC$ be a Borel barycenter map as in lemma \ref{L-Borel barycenter map}.
		\begin{enumerate}        
			\item If the homogeneous multimarginal problem has a solution $ \boldsymbol{\alpha}\in \cM_2(\fC^p) $ then $ \mathfrak{h}^2T_\#\boldsymbol{\alpha} $ is a barycenter of $\mu_1,\dots,\mu_p$ with weights $\lambda_1,\dots,\lambda_p$;
			\item Assume further that for every $\boldsymbol{\eta}=(\eta_1,\dots,\eta_p)\in \fC^p$, the point $T(\boldsymbol{\eta})$ is the unique barycenter of $(\eta_i,\lambda_i)$. Then every Hellinger-Kantorovich barycenter $\mu_0$ of $(\mu_i,\lambda_i)$ must have the form $ \mathfrak{h}^2T_\#\boldsymbol{\alpha} $ for some solution $ \boldsymbol{\alpha}\in \cP_2(\fC^p) $ of the homogeneous multimarginal problem with $\boldsymbol{\alpha}$ being concentrated on $ \mathfrak{C}^{p}[\boldsymbol{\Xi}] $, with $ \boldsymbol{\Xi}=\sqrt{\mu_0(X)}+\sum_{i=1}^pHK(\mu_0,\mu_i) $.
		\end{enumerate}	
	\end{thm}
	\begin{proof}
		(1)	Let $ \mu^n\in \cM(X) $ be a sequence minimizing $ \sum_{i=1}^p\lambda_iHK^2(\mu,\mu_i) $ and $ \boldsymbol{\alpha}\in \cM_2(\fC^p) $ be a minimizer of the homogeneous multimarginal problem. Then from lemma \ref{L-lifting gluing measure}, for every $n\in \N$
		choosing $\boldsymbol{\beta}^n\in \cP_2(\fC^{p+1})$ with 
		$\mathfrak{h}_{1}^2\boldsymbol{\beta}^n=\mu^n, \mathfrak{h}_{i+1}^2\boldsymbol{\beta}^n=\mu_i$ and $HK^2(\mu^n,\mu_i)=\int d^2_\fC(\eta_0,\eta_i)d\boldsymbol{\beta}^n, i=1,\dots, p,$
		we get that
		\begin{align*}
		\sum_{i=1}^p\lambda_iHK^2(\mu^n,\mu_i)=&\int \sum_{i=1}^p\lambda_id_\mathfrak{C}^2(\eta_0,\eta_i)d\boldsymbol{\beta}^n\\
		\ge&\int\sum_{i=1}^p \lambda_id_\mathfrak{C}^2(T(\eta_1,\ldots,\eta_p),\eta_i)d\boldsymbol{\beta}^n\\
		\ge&\int\sum_{i=1}^p \lambda_id_\mathfrak{C}^2(T(\eta_1,\ldots,\eta_p),\eta_i)d\boldsymbol{\alpha}\\
		\ge&\sum_{i=1}^p\lambda_iHK^2(\mathfrak{h}^2T_\#\boldsymbol{\alpha},\mu_i).
		\end{align*}
		Thus, $ \mathfrak{h}^2T_\#\boldsymbol{\alpha} $ is a barycenter of $\mu_1,\dots, \mu_p$ with respect to $\lambda_1,\dots,\lambda_p$.
		
		(2)	Let $ \mu_0 $ be a barycenter of $\mu_1,\dots, \mu_p$ with respect to $\lambda_1,\dots,\lambda_p$ then from lemma \ref{L-lifting gluing measure} there is $ \boldsymbol{\beta}\in\cP_2(\fC^{p+1}) $ concentrated on $ \mathfrak{C}^{p+1}[\boldsymbol{\Xi}] $, with $ \boldsymbol{\Xi}=\sqrt{\mu_0(X)}+\sum_{i=1}^pHK(\mu_0,\mu_i) $ such that
		\begin{align*}
		\sum_{i=1}^{p}\lambda_iHK^2(\mu_0,\mu_i)=&\int\sum_{i=1}^{p}\lambda_i d^2_{\mathfrak{C}}(\eta_0,\eta_i)d\boldsymbol{\beta}(\eta_0,\ldots,\eta_p)\\
		\ge&\int\sum_{i=1}^{p}\lambda_i d^2_{\mathfrak{C}}(T(\eta_1,\ldots,\eta_p),\eta_i)d\boldsymbol{\beta}(\eta_0,\ldots,\eta_p)\\
		=&\int\sum_{i=1}^{p}\lambda_i d^2_{\mathfrak{C}}(T(\eta_1,\ldots,\eta_p),\eta_i)d\Pi_{1,\ldots,p\#}\boldsymbol{\beta}(\eta_1,\ldots,\eta_p)\\
		\ge&\sum_{i=1}^{p}\lambda_iHK^2(\mathfrak{h}^2T_\#\Pi_{1,\ldots,p\#}\boldsymbol{\beta},\mu_i).
		\end{align*}
		Since all the equations hold, $ \Pi_{1,\ldots,p\#}\boldsymbol{\beta} $ must be a minimizer for homogeneous multimarginal problem and $ \boldsymbol{\beta} $ is concentrated on the graph $ (T(\eta_1,\ldots,\eta_p),\eta_1,\ldots,\eta_p). $ Thus, $ \mu_0=\mathfrak{h}^2T_\#\boldsymbol{\alpha}, $ where $\boldsymbol{\alpha}=\Pi_{1,\ldots,p\#}\boldsymbol{\beta}\in \cP_2(\fC^p)$.
	\end{proof}
	As a consequence, we get a consistency result for Hellinger-Kantorovich barycenters. It has been proved for Wasserstein spaces in \cite[Theorem 3.1]	{MR3338645}.
	\begin{thm}
		\label{T-consistency}
		Let $(X,d)$ be a Polish metric space having property $(BC)$. Let $T:\fC^p\to \fC$ be a Borel barycenter map as in lemma \ref{L-Borel barycenter map}. Assume that $T$ is continuous and for every $\boldsymbol{\eta}=(\eta_1,\dots,\eta_p)\in \fC^p$, the point $T(\boldsymbol{\eta})$ is the unique barycenter of $(\eta_i,\lambda_i)$.
		Let $ \{\mu_1^n\},\dots,\{\mu_p^n\} $ be sequences in $\cM(X)$ such that $ \mu_i^n $ converges to $ \mu_i $ weakly*, and for every $n\in \N$ let $ \mu_B^n $ be any Hellinger-Kantorovich barycenter of $ (\mu_i^n,\lambda_i)$. Then the sequence $\{ \mu_B^n\}_n $ is relative compact and any of its limit is a Hellinger-Kantorovich barycenter for $ (\mu_i,\lambda_i) $.
	\end{thm}
	\begin{proof}
		According to theorem \ref{bary+multi}, there exists $ \boldsymbol{\alpha^n}\in \cP_2(\fC^p) $ which is a solution for homogeneous multimarginal problem and satisfying $ \mathfrak{h}^2T_\#\boldsymbol{\alpha^n}=\mu_B^n $ for each $ n $. Also, $ \boldsymbol{\alpha}^n $ is concentrated on $ \mathfrak{C}^p[\sqrt{\mu_B^n(X)}+\sum_{i=1}^pHK(\mu_B^n,\mu_i^n)] $ and $ T_\#\boldsymbol{\alpha}^n $ is concentrated on $ \mathfrak{C}[\sqrt{\mu_B^n(X)}+\sum_{i=1}^pHK(\mu_B^n,\mu_i^n)] $. Let $\nu_0$ be the null measure on $X$. Then we have $HK^2(\nu,\nu_0)=\nu(X)$ for every $\nu\in \cM(X)$. Therefore
		\begin{align*}
		\dfrac{1}{2}\mu_B^n(X)&=\dfrac{1}{2}HK^2(\mu_B^n,\nu_0)\\
		&\le\dfrac{1}{2}\sum_{i=1}^p\lambda_i[ HK(\mu_i^n,\nu_0)+HK(\mu_B^n,\mu_i^n)]^2\\
		&\le\sum_{i=1}^p\lambda_i\mu_i^n(X)+\sum_{i=1}^p\lambda_iHK^2(\mu_B^n,\mu_i^n).
		\end{align*}
		We also have 
		$$ \sum_{i=1}^pHK(\mu_B^n,\mu_i^n)\le\sum_{i=1}^p\dfrac{1}{4\lambda_i}+\sum_{i=1}^p\lambda_iHK^2(\mu_B^n,\mu_i^n).$$
		Because $ \mu_i^n $ converges to $ \mu_i $ as $n\to \infty$, we get that $ \{\mu_i^n(X)\}_n $ is bounded. Using \cite[Theorem 7.15]{MR3763404}) we also get that $\{ \sum_{i=1}^p\lambda_iHK^2(\mu_B^n,\mu_i^n) \}_n$ is bounded. Hence $\{ \sqrt{\mu_B^n(X)}+\sum_{i=1}^pHK(\mu_B^n,\mu_i^n) \}_n$ is bounded above by some $ \boldsymbol{\Xi}>0 $. Then we can assume that $ \boldsymbol{\alpha}^n $ is concentrated on $ \mathfrak{C}^p[\boldsymbol{\Xi}] $, $ T_\#\boldsymbol{\alpha}^n $ is concentrated on $ \mathfrak{C}[\boldsymbol{\Xi}] $ for every $n\in \N$. As $ \{\mu_i^n \}_n$ is a convergent sequence we obtain that $ \{\mu_i^n\}_n $ is an equally tight set in $\cM(X)$ then $ \{\boldsymbol{\alpha}^n\} $ is also an equally tight set in $\cM(\fC^p)$ by \cite[Lemma 7.3]{MR3763404}). As $\boldsymbol{\alpha}^n $ is also a probability measure for every $n$, we get that $\{\boldsymbol\alpha^n\}_n $ is relative compact. Because $T$ is continuous and $ T_\#\boldsymbol{\alpha}^n $ is concentrated on $ \mathfrak{C}[\boldsymbol{\Xi}] $ for every $n\in \N$, we get that the sequence $\{ \mu_B^n\} $ is also relative compact.
		
		Take a convergent subsequence of $ \mu_B^n $, to simplify, we still denote it by $ \mu_B^n $. We denote its limit by $ \mu_B $. Let $ \mu_0 $ be a Hellinger-Kantorovich barycenter for $( \mu_i,\lambda_i)$, then applying again \cite[Theorem 7.15]{MR3763404} we get
		\[ \sum_{i=1}^p\lambda_iHK^2(\mu_0,\mu_i^n)\rightarrow\sum_{i=1}^p\lambda_iHK^2(\mu_0,\mu_i), \]
		\[ \sum_{i=1}^p\lambda_iHK^2(\mu_B^n,\mu_i^n)\rightarrow\sum_{i=1}^p\lambda_iHK^2(\mu_B,\mu_i). \]
		On the other hand, \[ \sum_{i=1}^p\lambda_iHK^2(\mu_0,\mu_i^n)\ge\sum_{i=1}^p\lambda_iHK^2(\mu_B^n,\mu_i^n). \]
		Therefore, $ \sum_{i=1}^p\lambda_iHK^2(\mu_0,\mu_i)\ge\sum_{i=1}^p\lambda_iHK^2(\mu_B,\mu_i) $ which means $ \mu_B $ is a Hellinger-Kantorovich barycenter for $ \mu_1,\ldots,\mu_p $ with respect to $\lambda_1,\dots, \lambda_p$.
	\end{proof}
	When the space $\fC$ is NPC it is known that for every $\boldsymbol{\eta}=(\eta_1,\dots,\eta_p)\in \fC^p$ and $\lambda_1,\dots,\lambda_p\geq 0$ with $\sum_{i=1}^p\lambda_i=1$, the barycenter of $(\eta_i,\lambda_i)$ is unique \cite[Proposition 4.3]{Sturm} and furthermore the map $T:\fC^p\to \fC$ as in lemma \ref{L-Borel barycenter map} is continuous by \cite[Theorem 6.3]{Sturm}. As we know that $X$ is $\CAT(1)$ if and only if $\fC$ is $\CAT(0)$, theorems \ref{bary+multi} and \ref{T-consistency} work for all Polish $\CAT(1)$ spaces.   
\begin{remark}
Let $n\geq 2$ and $X=\mathbb{S}^{n-1}\subset \R^n$ be the $(n-1)$-sphere endowed with the standard spherical angular metric $d$, i.e. $d(x,y):=\arccos\big(1-\frac{1}{2}\|x-y\|^2\big)$ for every $x,y\in X$. Then its Euclidean cone $(\fC,d_\fC)$ is the usual Euclidean space $(\R^n,\|\cdot\|)$. In this case, for every $p\geq 1$ and every weights $(\lambda_1,\cdots, \lambda_p)$, from \cite[page 913]{MR2801182} we know that the map $T:(\R^n)^p\to \R^n$ defined by $T(x_1,\cdots,x_p):=\sum_{i=1}^p \lambda_ix_i$ for every $(x_1,\cdots,x_p)\in (\R^n)^p$ is characterized by the property $$\inf_{y\in \R^n}\sum_{i=1}^p\lambda_i\|x_i-y\|^2=\sum_{i=1}^p\lambda_i\|x_i-T(x_1,\cdots,x_p)\|^2.$$  
Therefore, theorem \ref{bary+multi} provides us an explicit homogeneous marginal formulation of the Hellinger-Kantorovich barycenter problem in this case.
\end{remark}		
	
\end{document}